\documentclass[11pt]{amsart}

\usepackage[margin=1in]{geometry}

\usepackage{amssymb}
\usepackage{amsthm}
\usepackage{amsmath}
\usepackage[shortlabels]{enumitem}

\usepackage{hyperref}

\usepackage{url}

\usepackage{stmaryrd}

\let\savelneq\lneq
\let\lneq\relax
\usepackage{mathabx}
\let\lneq\savelneq

\usepackage{scalerel}

\newtheorem{thm}{Theorem}

\newtheorem{theorem}[thm]{Theorem}
\newtheorem{lemma}[thm]{Lemma}
\newtheorem{proposition}[thm]{Proposition}

\theoremstyle{definition}
\newtheorem{definition}[thm]{Definition}
\newtheorem{remark}[thm]{Remark}
\newtheorem{example}[thm]{Example}

\newtheorem{corollary}[thm]{Corollary}

\newtheorem{question}[thm]{Question}
\newtheorem{notation}[thm]{Notation}

\newtheorem*{bijectionExists}{Theorem \ref{thm:bijectionExists}}
\newtheorem*{rjust}{Theorem \ref{thm:rjust}}

\numberwithin{thm}{section}

\newcommand{\Slide}{\mathrm{Slide}}
\newcommand{\slide}{\mathrm{slide}}
\newcommand{\Tour}{\mathrm{Tour}}
\newcommand{\Cat}{\mathrm{Cat}}
\newcommand{\Av}{\mathrm{Av}}
\newcommand{\tree}{\mathrm{Tree}}
\newcommand{\word}{\mathrm{word}}
\newcommand{\z}{\mathrm{z}}
\newcommand{\rep}{\mathrm{BigRep}}

\newcommand{\last}{\mathrm{last}}
\newcommand{\PF}{\mathrm{PF}}
\newcommand{\CPF}{\mathrm{CPF}}
\newcommand{\Br}{\mathrm{Br}}
\newcommand{\rev}{\mathrm{rev}}
\newcommand{\red}{\mathrm{red}}

\newcommand{\PP}{\mathbb{P}}

\newcommand{\Mbar}{\overline{M}}
\newcommand{\kUnd}{\underline{k}}

\newcommand{\lnd}{\mathrm{Lnd}}
\newcommand{\rnd}{\mathrm{Rnd}}

\newcommand{\TRep}{\mathrm{TotalRep}}
\newcommand{\Slwk}{\Slide^\omega(\kUnd)}
\newcommand{\Slpk}{\Slide^\psi(\kUnd)}
\newcommand{\Slwl}{\Slide^\omega(1,1,\ldots,1)}
\newcommand{\maxzero}{\mathrm{maxzero}}
\newcommand{\mintwo}{\mathrm{min_2}}

\newcommand{\sigmaT}{\hat{\sigma}}
\newcommand{\piT}{\hat{\pi}}

\newcommand{\starTree}{\raisebox{-1pt}{\scaleobj{1.5}{\Ydown}}}
\newcommand{\starThree}{\raisebox{-1pt}{\scaleobj{1.5}{\convolution}}}

\newcommand{\mbf}[1]{\mathbf{#1}}
\def\multichoose#1#2{\left<\genfrac{}{}{0pt}{}{#1}{#2}\right>}

\usepackage{tikz}

\title[Insertion algorithms on trees in $\Mbar_{0,n+3}$]{Insertion algorithms and pattern avoidance on trees arising in the Kapranov embedding of $\Mbar_{0,n+3}$}
\author{Andrew Reimer-Berg}
\address{Andrew Reimer-Berg \\ Department of Mathematics \\Colorado State University \\ Fort Collins, CO 80523 \\ United States of America}
\email{\href{mailto:Andrew.Reimer-Berg@colostate.edu}{Andrew.Reimer-Berg@colostate.edu}}
\thanks{The author was partially supported by NSF DMS award number 2054391.}
\date{\today}

\begin{document}

	\begin{abstract}
		
		We resolve a question of Gillespie, Griffin, and Levinson that asks for a combinatorial bijection between two classes of trivalent trees, \textit{tournament trees} and \textit{slide trees}, that both naturally arise in the intersection theory of the moduli space $\Mbar_{0,n+3}$ of stable genus zero curves with $n+3$ marked points.  Each set of trees enumerates the same intersection product of certain pullbacks of $\psi$ classes under forgetting maps.
		
		We give an explicit combinatorial bijection between these two sets of trees using an insertion algorithm.  We also classify the words that appear on the slide trees of caterpillar shape via pattern avoidance conditions. 
	\end{abstract}
	
	\maketitle
	
	\section{Introduction}
	
	For a positive integer $n$, let $\Mbar_{0,n+3}$ be the Deligne-Mumford moduli space of stable genus 0 curves with $n+3$ marked points labeled by $\{a,b,c,1,2,\ldots,n\}$.   The $i$th cotangent line bundle $\mathbb{L}_i$ is the line bundle whose fiber over $C\in\Mbar_{0,n+3}$ is the cotangent space of $C$ at the marked point $i$. The \textbf{$i$th psi class} $\psi_i$ is the first Chern class of $\mathbb{L}_i$. In other words, $\psi_i=c_1(\mathbb{L}_i)$.  The \textbf{$i$th omega class} $\omega_i$ is defined as the pullback of $\psi_i$ under the composition of the forgetting maps that forget the marked points $i+1,\ldots,n$. 
	
	Let $\kUnd=(k_1,k_2,\ldots,k_n)$ be a $n$-tuple of nonnegative integers and assume it is a composition of $n$, that is, $k_1+k_2+\cdots+k_n=n$. Products in the cohomology ring of $\Mbar_{0,n+3}$ of the form $\psi^{\kUnd}:=\psi_1^{k_1}\cdots\psi_n^{k_n}$ and $\omega^{\kUnd}:=\omega_1^{k_1}\cdots\omega_n^{k_n}$ were studied in \cite{GGL22,GGL23} and shown to be computable by way of enumerating certain classes of trees called \textbf{slide trees} $\Slwk$ and \textbf{tournament trees} $\Tour(\kUnd)$. 
	
	\begin{proposition}[From \cite{GGL22} and \cite{GGL23}]
		Let $\kUnd$ be a composition of $n$. Then,
		$$\int_{\Mbar_{0,n+3}}\psi^{\kUnd}=\binom{n}{k_1,k_2,\ldots,k_n}=|\Slpk|$$ and $$\int_{\Mbar_{0,n+3}}\omega^{\kUnd}=\multichoose{n}{k_1,k_2,\ldots,k_n}=|\Tour(\kUnd)|=|\Slwk|.$$
		\label{prop:combinedResults}
	\end{proposition}
	
	The coefficients in the second part of the last equality are the \textbf{asymmetric multinomial coefficients}, which we define in Section \ref{subsec:Asym}. Several recent papers \cite{CGM,GGL22,GGL23} have studied these asymmetric multinomial coefficients from both a geometric and combinatorial perspective.  This forms part of a growing body of work that has been done using combinatorial objects to study the intersection theory of $\Mbar_{0,n}$. In addition to the above works, other work studying the degrees of projective maps on moduli spaces of curves include \cite{silversmith21}, working in terms of cross-ratio degrees, and \cite{kapranovdegrees}, which uses a different class of trees to study more general pullbacks of $\psi$ classes. Work on products of $\psi$ classes in tropical $M_{0,n}$ include \cite{tropical_Hassett,kerber_markwig}.
	
	The sets $\Tour(\kUnd)$ and $\Slwk$ are both sets of trivalent trees whose leaves are labeled by the set $\{a,b,c,1,2,\ldots,n\}$. Despite this similarity, finding a combinatorial bijection between these two sets has until now been an open question. (See Problem 6.1 in \cite{GGL22}).
	Our first main result constructs an explicit bijection between $\Tour(\kUnd)$ and $\Slwk$, thus proving combinatorially that $|\Tour(\kUnd)|=|\Slwk|$, a fact that was previously only known through geometric techniques.
	
	The set $\CPF(\kUnd)$ was the first combinatorial interpretation, in terms of \textit{parking functions} \cite{CGM}.  A bijection between $\Tour(\kUnd)$ and $\CPF(\kUnd)$ is given in \cite{GGL23}.  Both the parking functions and tournaments interpretations were shown to satisfy the asymmetric multinomial recursion \eqref{eq:recursion} defined in Section \ref{subsec:Asym}.
	
	We similarly build our bijection recursively, with the main step being to show that $|\Slide(\kUnd)|$ also satisfies the same recursion. In particular, we build a bijection between $\Slwk$ and a disjoint union of slide sets $\Slide^\omega(\kUnd^{(j)})$ for compositions $\kUnd^{(j)}$ of $n-1$, via an insertion algorithm on $\Slwk$.  Then, we can unwind the recursive algorithms for each of $\Slide(\kUnd)$ and $\Tour(\kUnd)$ to recover a full bijection $F:\Slide(\kUnd)\to \Tour(\kUnd)$.
	
	\begin{theorem}
		The map $F$ is a combinatorial bijection between the sets $\Tour(\kUnd)$ and $\Slwk$.
		\label{thm:bijectionExists}
	\end{theorem}
	
	Our second main result is on caterpillar trees.  We say a trivalent tree is a \textbf{caterpillar tree} if its internal edges form a path (see Example \ref{ex:slides} below).  We can form a word from a caterpillar tree by reading the slide labels defined in Section \ref{sec:slides}, and obtain the following pattern avoidance condition that generalizes results in \cite{GGL22}.
	
	We have a map $\tree$ that, given a word, constructs the slide tree of caterpillar shape whose edge labels, as read off from the root, form the given word, if such a tree exists. Otherwise, it forms some caterpillar tree that is not a valid slide tree. That is, a word $w$ can be read off of the edges of a caterpillar slide tree if and only if $\tree(w)$ is a valid slide tree. Our result then characterizes caterpillar trees by way of characterizing these words $w$.
	
	\begin{theorem}
		Let $\kUnd$ be a right-justified composition of $n$. Then, $\tree(w)$ is a valid slide tree if and only if $w$ avoids the patterns $2{-}1{-}2$ and $23{-}\overline{2}{-}1$.  (See Section \ref{sec:patterns}.)
		\label{thm:rjust}
	\end{theorem}
	Above, a \textit{right-justified} composition is a composition where all non-zero entries appear right of any 0 entries.	
	Otherwise, if $\kUnd$ is not right-justified, the set of caterpillar trees can not be characterized solely by a pattern avoidance criterion.  We state the characterization result below and define the terms precisely in Section \ref{sec:cats}.
	
	\begin{theorem}
		Let $\kUnd$ be a reverse-Catalan composition of $n$, and let $w$ be a word of content $\kUnd$.  Then the caterpillar tree $\tree(w)$ is in $\Slpk$ (resp.,  $\tree(w)$ is in $\Slwk$) if and only if $w$ avoids the patterns $2{-}1{-}2$ and $23{-}\overline{2}{-}1$, and the inequality $$\TRep_w(i)+\ell_w(i)\geq \z(i)$$ holds for all $i$ (respectively, $\rep_w(i)\geq \z(i)$ holds for all $i$).  
		
		Here, $\ell_w(i)$ is the number of $i$s in the rightmost consecutive group of $i$s in $w$, $\TRep_w(i)$ is the total number of repeated letters to the right of the rightmost $i$ in $w$, and $\rep_w(i)$ is the total number of repeated letters larger than $i$ to the right of the rightmost $i$ in $w$.   
		\label{thm:psiAvoidance}
	\end{theorem}

	Our work is structured as follows.	In Section 2 we recall the necessary geometry and combinatorics background and some known results about asymmetric multinomial coefficients and slide trees.  In Section 3 we explore the combinatorics of slide trees and introduce notions that we will use throughout the rest of the paper. In Section 4 we define our insertion algorithms for $F$ and prove Theorem \ref{thm:bijectionExists}. In Section 5, we prove Theorems \ref{thm:rjust} and \ref{thm:psiAvoidance}. In Section 6, we explore the case $\kUnd=(1,1,\ldots,1)$, where the bijection can be given more directly.
	
	\subsection{Acknowledgments}
	We thank Maria Gillespie for her mentorship, and thank
	Vance Blankers, Renzo Cavalieri, Sean Griffin, Matt Larson, and Jake Levinson for their helpful conversations.

	\section{Background}
	
	We start by defining the relevant geometric objects of study. 
	
	\subsection{The moduli space of curves $\Mbar_{0,n}$}
	\label{sec:Mbar}
	
	Broadly speaking, a moduli space is a geometric space whose points parameterize some class of geometric objects. In our case, these objects are complex curves with marked points. 
	
	\begin{definition}
		Let $M_{0,n}$ be the moduli space of (complex) genus 0 curves with $n$ marked points. In particular, a curve $C\in M_{0,n}$ is isomorphic to $\PP^1$ and consists of the data of $n$ distinct points $p_1,p_2,\ldots,p_n\in\PP^1$.
	\end{definition}
	
	\begin{definition}
		Then, let $\Mbar_{0,n}$ be the \textbf{Deligne-Mumford compactification} of $M_{0,n}$, the moduli space of \textit{stable} genus 0 curves with $n$ marked points. A boundary curve $C\in\Mbar_{0,n}\setminus M_{0,n}$ will instead consist of two or more components, each individually isomorphic to $\PP^1$, joined together at nodes in a tree structure (to keep the total genus 0), with the $n$ marked points distributed between them.  (See Figure \ref{fig:MbarEx}.)
	\end{definition}
	
	We will often refer to $M_{0,n}$ as the \textit{interior} of $\Mbar_{0,n}$, and $\Mbar_{0,n}\setminus M_{0,n}$ as the \textit{boundary}. 
	We use the term \textbf{special point} to mean either a marked point or a node. The word \textit{stable} in the above definition means that each component has at least three special points.
	
	For more details on the construction and properties of $\Mbar_{0,n}$, we refer the reader to \cite{Cavalieri16} or \cite[Chapter~1]{green_book}.
	
	\begin{remark}
		Throughout this work, as in \cite{GGL22}, we label our marked points by the alphabet $\{a,b,c,1,2,\ldots,n\}$, and thus will henceforth refer to $\Mbar_{0,n}$ as $\Mbar_{0,n+3}$ instead.
		
		When we wish to apply an ordering to this alphabet, we will impose that $a<b<c<1<2<\cdots<n$.
	\end{remark}
	
	Given a stable curve $C\in\Mbar_{0,n+3}$, we may also consider its \textbf{dual tree}. To form the dual tree to a curve, create a vertex for each component of $C$, as well as one for every marked point. Then, for each marked point, connect its vertex to the vertex corresponding to the component it is contained in. For each node, connect the vertices corresponding to the two components that intersect at that node.
	
	\begin{example}
		On the left of Figure \ref{fig:MbarEx} is an example of a boundary curve in $\Mbar_{0,7}$ (so $n=4$). To its right, we construct its dual tree.
	\end{example}
	
	\begin{figure}[bt]
		\centering
		\includegraphics{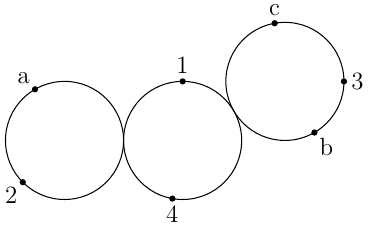}
		\hspace{2cm}
		\includegraphics{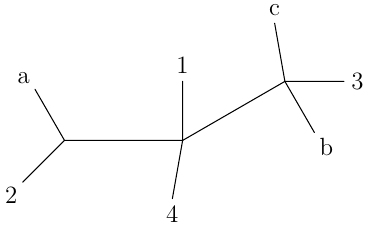}
		\caption{An element of $\Mbar_{0,n+3}$ for $n=4$, along with its dual tree.}
		\label{fig:MbarEx}
	\end{figure}
	
	The boundary $\Mbar_{0,n+3}\setminus M_{0,n+3}$ can be divided into different \textbf{boundary strata}. A particular stratum consists of all curves in $\Mbar_{0,n+3}$ that share the same dual tree.
	The interior $M_{0,n+3}$ is a single stratum, as all interior curves have the star graph as their dual tree.
	
	We now define forgetting maps on $\Mbar_{0,n+3}$. 
	
	\begin{definition}
		Define the \textbf{$n$th forgetting map} $\pi_n: \Mbar_{0,n+3}\to\Mbar_{0,n+2}$ as the map that forgets the marked point $n$. If this results in a component having less than three special points, we must \textit{stabilize} the curve. To do this, if forgetting $n$ results in a component with two nodes and no marked points, replace that component with just a node. If this results in a component with one node and one marked point, replace it with that marked point.
	\end{definition}

	Note that we may compose these maps as desired: $\pi_{n-2}\circ\pi_{n-1}\circ\pi_n:\Mbar_{0,n+3}\to\Mbar_{0,n},$ and so forth.
	
	\subsection{Psi and omega classes, the Kapranov embedding, and multidegrees}
	
	We next define certain classes in the cohomology of $\Mbar_{0,n+3}$, as well as the corresponding map to projective space. For full details, see \cite{Kapranov}, where Kapranov first defined this map.
	
	The $i$th cotangent line bundle $\mathbb{L}_i$ is the line bundle whose fiber over $C\in\Mbar_{0,n+3}$ is the cotangent space of $C$ at the marked point $i$. The \textbf{$i$th psi class} $\psi_i$ is the first Chern class of $\mathbb{L}_i$. In other words, $\psi_i=c_1(\mathbb{L}_i)$.
	
	Next, the \textbf{$i$th omega class} $\omega_i$ is defined as the pullback of $\psi_i$ under the forgetting maps that forget the marked points $i+1,\ldots,n$. 
	
	We also consider the corresponding maps to projective space: The Kapranov morphism $|\psi_i|:\Mbar_{0,n+3}\to\PP^n$, and $|\omega_i|:\Mbar_{0,n+3}\to\PP^i$, given by $|\omega_i|=|\psi_i|\circ\pi_{i+1}\circ\cdots\circ\pi_n$.
	
	\begin{example}
		Let $C$ be the curve in $\Mbar_{0,7}$ seen in Figure \ref{fig:MbarEx}. Suppose that on the central component $C'$, the right node has coordinate $s$, and the point 4 has coordinate $t$. Then, $|\psi_1|(C)=[s:s:0:s:t]$ and $|\omega_1|(C)=[s:s]$.
	\end{example}
	
	We will now combine these Kapranov morphisms together to form two different maps from $\Mbar_{0,n+3}$ into products of projective spaces. The first is the \textbf{total Kapranov map} $\Psi_n:\Mbar_{0,n+3}\to \PP^n\times\PP^n\times\cdots\times\PP^n$ given by $$\Psi_n(C)=(|\psi_1|(C),|\psi_2|(C),\ldots,|\psi_n|(C)).$$
	
	However, $\Psi_n$ is not an embedding, as it hides information about any components of $C$ that contain no marked points. 
	To embed $\Mbar_{0,n+3}$ into a product of projective spaces, we instead use the \textbf{iterated Kapranov map} $\Omega_n$. This map $\Omega_n: \Mbar_{0,n+3}\hookrightarrow\PP^1\times\PP^2\times\cdots\times\PP^n$ is defined by $$\Omega_n(C)=(|\omega_1|(C),|\omega_2|(C),\ldots, |\omega_n|(C)).$$ The iterated Kapranov map wasstudied in more detail in \cite{KeelTevelev,MoninRana17}.

	\begin{definition}
		Let $\kUnd=(k_1,k_2,\ldots,k_n)$ be a $n$-tuple of nonnegative integers. Then $\kUnd$ is a \textbf{composition of $n$} if $k_1+k_2+\cdots+k_n=n$.
	\end{definition}
	
	\begin{definition}
		Let $\kUnd$ be a composition of $n$.
		Then, consider the intersection between the image of $\Psi_n$ (resp. $\Omega_n$) with $n$ hyperplanes, where we choose $k_i$ general hyperplanes from the $i$th component of the product. Then, the \textbf{multidegree} of $\Psi_n$ (resp. $\Omega_n$) is the number of points in this intersection. This is denoted as $\deg_{\kUnd}(\Psi_n)$ (resp. $\deg_{\kUnd}(\Omega_n)$).
	\end{definition}
	
	It is known (see \cite{Cavalieri16}, for example) that when $\sum k_i=n$, $$\deg_{\kUnd}(\Psi_n)=\int_{\Mbar_{0,n+3}}\psi^{\kUnd}=\binom{n}{k_1,k_2,\ldots,k_n},$$ where $\psi^{\kUnd}$ denotes the product $\prod_i \psi_i^{k_i}$. 
	
	A similar result for omega classes also exists. It is shown in \cite{CGM} that when $\sum k_i=n$, $$\deg_{\kUnd}(\Omega_n)=\int_{\Mbar_{0,n+3}}\omega^{\kUnd}=\multichoose{n}{k_1,k_2,\ldots,k_n}.$$ The coefficients in the third part of the above equality are the \textit{asymmetric multinomial coefficients}, which we define in the next subsection.
	
	\subsection{Asymmetric multinomial coefficients}
	\label{subsec:Asym}
	
	The asymmetric multinomial coefficients were originally defined in \cite{CGM}, whose definition we restate here, using modified notation that matches the indexing used in \cite{GGL23}.
	
	\begin{definition}
		Let $\kUnd$ be a composition of $n$. Let $k_i$ be the rightmost 0 in $\kUnd$ (we set $i=0$ if there are no zeroes in $\kUnd$), and let $j>i$ be a positive integer. Define $\kUnd^{(j)}$ to be the composition of $n-1$ formed by decreasing $k_j$ by 1 and then removing the rightmost 0 (which is either in position $j$ or $i$) from the resulting tuple. Then, the \textbf{asymmetric multinomial coefficients $\multichoose{n}{\kUnd}$} are defined by $\multichoose11=1$ and the recursion \begin{equation}\multichoose{n}{\kUnd}=\sum_{j=i+1}^{n}\multichoose{n-1}{\kUnd^{(j)}}.\label{eq:recursion}\end{equation}
		\label{def:asymb}
	\end{definition}
	
	Notably, if $\kUnd=(1,1,\ldots,1)$, then $\multichoose{n}{\kUnd}=n!$, and if $\kUnd=(0,\ldots,0,n)$, then $\multichoose{n}{\kUnd}=1$.
	
	\begin{example}
		We compute the coefficient $\multichoose{4}{1,0,2,1}$: \begin{align*} \multichoose{4}{1,0,2,1} &= \multichoose{3}{1,1,1}+\multichoose{3}{1,0,2}\\ &= 3!+\multichoose{2}{1,1}\\ &= 3!+2!=8. \end{align*}
	\end{example}
	
	\begin{example}
		We compute the coefficient $\multichoose{4}{0,1,2,1}$: \begin{align*} \multichoose{4}{0,1,2,1} &= \multichoose{3}{0,2,1}+\multichoose{3}{1,1,1}+\multichoose{3}{0,1,2}\\ &= \multichoose{2}{1,1}+\multichoose{2}{0,2}+3!+\multichoose{2}{0,2}+\multichoose{2}{1,1}\\ &= 2!+1+3!+1+2!=12. \end{align*}
		Note that in this case, the asymmetric multinomial $\multichoose{4}{0,1,2,1}$ is equal to the corresponding symmetric multinomial coefficient $\binom{n}{0,1,2,1}$. (In fact, this will happen any time that $\kUnd$ is \textit{right-justified}, that is, where all entries that are zero occur before all non-zero entries.)
	\end{example}

	\begin{definition}
		Let $\kUnd=(k_1,k_2,\ldots,k_n)$ be a composition of $n$. We say that $\kUnd$ is \textbf{reverse-Catalan} if for all $i$, $k_{n-i+1}+\cdots+k_{n-1}+k_n\geq i$.
	\end{definition}
	
	\begin{proposition}[Corollary 4.14 from \cite{CGM}]
		The coefficient $\multichoose{n}{\kUnd}\neq0$ if and only if $\kUnd$ is reverse Catalan.
		\label{prop:revCatalan}
	\end{proposition}
	
	Several different combinatorial objects are counted by the asymmetric multinomial coefficients. The first one that we discuss are a variant of parking functions called \textit{column-restricted parking functions}. First, we define parking functions.
	
	\begin{definition}
		A \textbf{Dyck path} is a lattice path from $(0,n)$ to $(n,0)$ consisting of $n$ unit-length right steps and $n$ unit-length down steps, that stays weakly above the diagonal $y=n-x$. A \textbf{parking function} is a labeling of the down steps of a Dyck path with the numbers $1,2,\ldots,n$, such that in each column, the labels increase from bottom to top. We call $n$ the \textbf{semilength} of a parking function, and denote the set of semilength $n$ parking functions by $\PF(n)$.
	\end{definition}
	
	We label the columns of a parking function $P$ from left to right. We now restate the definitions of \textit{dominance} and of \textit{column-restricted parking functions} found in \cite{CGM} and \cite{GGL23}.
	
	\begin{definition}
		Let $P\in\PF(n)$ be a parking function. For a label $x\in[n]$ in column $j$ of $P$, we say $x$ \textbf{dominates} a column $i>j$ if $x$ is larger than every entry of $i$. (This includes both empty columns and columns with only entries smaller than $x$.) We define the \textbf{dominance index} $d_P(x)$ of $x$ to be the number of columns to the right of $x$ dominated by $x$.
	\end{definition}
	
	\begin{definition}
		Let $P$ be a parking function where $d_P(x)<x$ for all $x$. Then $P$ is a \textbf{column-restricted parking function}. We denote the set of such parking functions by $\CPF(n)\subseteq\PF(n)$.
	\end{definition}
	
	We may also write $\PF(\kUnd)$ or $\CPF(\kUnd)$ to denote the set of parking functions or column-restricted parking functions with $k_i$ down-steps in column $i$, respectively.
	
	\begin{example}
		Below is an example of a parking function $P\in\PF(0,1,1,0,1,3)$. The dominance indices of $P$ are $d_P(1)=0$, $d_P(2)=2$, $d_P(3)=0$, $d_P(4)=3$, $d_P(5)=0$, and $d_P(6)=0$. Since $d_P(2)\geq2$, $P$ is not a column-restricted parking function.

		\begin{center}
			\begin{tikzpicture}[scale=0.6]
			\draw (0,0)--(6,-6);
			\draw[thick] (0,0)--(2,0)--(2,-1)--(3,-1)--(3,-2)--(5,-2)--(5, -3)--(6,-3)--(6,-6);
			\node() at (1.75,-0.5) {$\mbf4$};
			\node() at (2.75,-1.5) {$\mbf2$};
			\node() at (4.75,-2.5) {$\mbf1$};
			\node() at (5.75,-3.5) {$\mbf6$};
			\node() at (5.75,-4.5) {$\mbf5$};
			\node() at (5.75,-5.5) {$\mbf3$};
			\end{tikzpicture}
		\end{center}
	\end{example}
	
	Secondly, \cite{GGL23} defines a set of trivalent trees $\Tour(\kUnd)$, called \textbf{tournament trees}, which are defined via a process called \textit{lazy tournaments} and are also counted by the asymmetric multinomial coefficients.
	Finally, there is a second class of trivalent trees that are also counted by the asymmetric multinomial coefficients. These are called \textbf{slide trees}, denoted $\Slwk$. These were originally defined in \cite{GGL22}, and we describe them in detail in the next subsection.
	Putting this all together, we have the following proposition, where the last three equalities come from \cite{CGM}, \cite{GGL23}, and \cite{GGL22}, respectively.
	
	\begin{proposition}
		Let $\kUnd$ be a composition of $n$. Then, $$\deg_{\kUnd}(\Omega_n)=\multichoose{n}{\kUnd}= |\CPF(\kUnd)|=|\Tour(\kUnd)|=|\Slwk|.$$
	\end{proposition}
	
	The third and fourth objects above are two different sets of trivalent trees. This leads naturally to the following combinatorial question:
	
	\begin{question}[Problem 6.1 from \cite{GGL22}]
		Find a combinatorial bijection between the sets $\Tour(\kUnd)$ and $\Slwk$.
		\label{q:mainQuestion}
	\end{question}

	\subsection{Slide trees}
	\label{sec:slides}
	We now recall the definitions of $\Slpk$ and $\Slwk$ given in \cite{GGL22}. 
	
	\begin{definition}
		A tree $T$ is \textbf{at least trivalent} or \textbf{stable} if every non-leaf vertex has degree at least 3.
	\end{definition}
	
	Comparing this to our definition of stable curves, we see that a curve $C$ is stable precisely when its dual tree is stable. Let $T$ be a stable tree with leaves labeled by $a<b<c<1<2<\ldots< n$. For $i\in[n]$, let $v_i$ be the vertex adjacent to the leaf $i$.
	
	\begin{definition}
		For a fixed $i\in[n]$, let $\Br_a$ be the branch at $v_i$ containing $a$, and $e_a$ the edge connecting $\Br_a$ to $v_i$. Let $m$ be the minimal leaf label of $T\setminus(\Br_a\cup\{i\})$, and $\Br_m$ be the branch at $i$ containing $m$.
		\label{def:pre-i-slide}
	\end{definition}

	\begin{definition}
		We define an \textbf{$i$-slide} on a stable tree $T$ as follows. Add a vertex $v'$ in the middle of edge $e_a$. Reattach the branch $\Br_m$ to be rooted at $v'$. Leave $\Br_a$ and the branch that is just $i$ as is. For all other branches of $T$ at $v_i$, either leave them rooted at $v_i$, or reattach them to be rooted at $v'$.
		Denote the set of trees obtained this way as $\slide_i(T)$.
		\label{def:i-slide}
	\end{definition}
	
	Note that in order for the result of an $i$-slide to be stable, at least one branch (besides $i$) must remain at $v_i$. In particular, if $\deg(v_i)=3$, then $\slide_i(T)=\emptyset$.
	
	\begin{definition}[$\psi$ slide rule]
		Define $\Slpk$ as the set of all stable trees obtained by the following process.
		\begin{enumerate}
			\item For step $i=0$, start with $\starTree$.
			\item For steps $i=1,\ldots,n$, perform $k_i$ $i$-slides to all trees obtained in step $i-1$ in all possible ways.
		\end{enumerate}
		
		\label{def:slideTreesPsi}
	\end{definition}
	
	\begin{definition}[$\omega$ slide rule]
		Define $\Slwk$ as the set of all stable trees obtained by the following process.
		\begin{enumerate}
			\item For step $i=0$, start with $\starThree$.
			\item For steps $i=1,\ldots,n$:
			\begin{enumerate}
				\item Consider the set of all trees formed by attaching the leaf $i$ to a non-leaf vertex of a tree obtained in step $i-1$ in all possible ways.
				\item Perform $k_i$ $i$-slides to all trees obtained in (a) in all possible ways.
			\end{enumerate}
		\end{enumerate}
		\label{def:slideTreesOmega}
	\end{definition}
	
	This definition is originally motivated by the work in \cite{GGL22}. In particular, consider the image of $\Mbar_{0,n+3}$ under the Kapranov maps in Section \ref{sec:Mbar} and intersect it with a particular class of hyperplanes in the product of projective space. If we then degenerate those hyperplanes in the right way and pull back through the Kapranov map, we get a union of 0-dimnensional boundary strata whose dual trees are precisely the trees in $\Slwk$.
	
	Alternatively, we define a procedure known as the $\kUnd$-labeling algorithm, which is a method for determining whether a given tree is in a given slide set. Theorem 3.14 from \cite{GGL22} states that a tree $T$ is in $\Slwk$ (resp. $\Slpk$) precisely if it admits an $\omega$ (resp. $\psi$) $\kUnd$-slide labeling. So, instead of computing an entire slide set using Definition \ref{def:slideTreesPsi} or \ref{def:slideTreesOmega}, we can instead use \ref{def:k-slide} to test whether a particular tree is in a particular slide set.
	
	\begin{definition}
		Define the $\omega$ (resp. $\psi$) $\kUnd$-slide labeling of a stable tree $T$ as the result of the following procedure, if it finishes. (Otherwise, the $\kUnd$-slide labeling does not exist.)
		\begin{enumerate}\setcounter{enumi}{-1}
			\item Start with $\ell=n$.
			\item \textbf{Contract labeled edges:} Let $T'$ be the tree formed from $T$ by contracting all internal labeled edges.
			\item \textbf{Identify the next edge to label:} Let $e$ be the first unlabeled internal edge on the path from the leaf $\ell$ to $a$. (If no such edge exists, then the process terminates and no $\kUnd$-slide labeling exists.) Let $v_\ell$ be the vertex adjacent to $\ell$, and $v_a$ be the other vertex of $e$.
			\item \textbf{Verify that label is valid:} Let $m_\ell$ be the smallest leaf label among all branches of $v_\ell$ not containing $a$ or $\ell$, and $m_a$ as the smallest leaf label among all branches of $v_a$ not containing $a$ or $\ell$.
			In the $\omega$ case (resp. $\psi$ case), if $\ell\geq m_\ell\geq m_a$ (resp. $m_\ell\geq m_a$), then label $e$ with $\ell$. Otherwise, the process terminates.
			\item \textbf{Iterate:}  If $\ell$ has labeled fewer than $k_\ell$ edges, repeat this process with the same $\ell$. Otherwise, decrement $\ell$. If $\ell=0$, then we have successfully constructed the $\kUnd$-slide labeling of $T$.
		\end{enumerate}
		\label{def:k-slide}
	\end{definition}
	
	\begin{example}
		\label{ex:slides}
		Consider the tree below on the left. As we perform the first round of the $\omega$ $(0,0,2,1,1,2)$-slide algorithm, we have $\ell=6$, and try to label the edge above the leaf 6. We have $m_\ell=1$ and $m_a=c$, so since $c<1<6$, we label this edge $\mbf 6$, as shown below on the left. Continuing this process, we end up with the labeling on the right.
		\begin{center}
			\begin{tikzpicture}[scale=1]
			\draw (0,0)--(5,0);
			\draw (0,0)--(-0.5,0.5);
			\draw (0,0)--(-0.5,-0.5);
			\draw (5,0)--(5.5,0.5);
			\draw (5,0)--(5.5,-0.5);
			\draw (1,0)--(1,-1);
			\draw (2,0)--(2,-0.5);
			\draw (3,0)--(3,-0.5);
			\draw (4,0)--(4,-0.5);
			\draw (1,-1)--(0.5,-1.5);
			\draw (1,-1)--(1.5,-1.5);
			
			\node()[above left] at (-0.4,0.4) {$a$};
			\node()[below left] at (-0.4,-0.4) {$b$};
			\node()[above right] at (5.4,0.4) {$3$};
			\node()[below right] at (5.4,-0.4) {$2$};
			\node()[below left] at (0.6,-1.4) {$6$};
			\node()[below right] at (1.4,-1.4) {$1$};
			\node()[below] at (2,-0.5) {$5$};
			\node()[below] at (3,-0.5) {$c$};
			\node()[below] at (4,-0.5) {$4$};
			
			\node()[left] at (1,-0.5) {$\mbf 6$};
			
			\end{tikzpicture}
			\hspace{0.6cm}
			\begin{tikzpicture}[scale=1]
			\draw (0,0)--(5,0);
			\draw (0,0)--(-0.5,0.5);
			\draw (0,0)--(-0.5,-0.5);
			\draw (5,0)--(5.5,0.5);
			\draw (5,0)--(5.5,-0.5);
			\draw (1,0)--(1,-1);
			\draw (2,0)--(2,-0.5);
			\draw (3,0)--(3,-0.5);
			\draw (4,0)--(4,-0.5);
			\draw (1,-1)--(0.5,-1.5);
			\draw (1,-1)--(1.5,-1.5);
			
			\node()[above left] at (-0.4,0.4) {$a$};
			\node()[below left] at (-0.4,-0.4) {$b$};
			\node()[above right] at (5.4,0.4) {$3$};
			\node()[below right] at (5.4,-0.4) {$2$};
			\node()[below left] at (0.6,-1.4) {$6$};
			\node()[below right] at (1.4,-1.4) {$1$};
			\node()[below] at (2,-0.5) {$5$};
			\node()[below] at (3,-0.5) {$c$};
			\node()[below] at (4,-0.5) {$4$};
			
			\node()[left] at (1.0,-0.5) {$\mbf 6$};
			\node()[above] at (0.5,0) {$\mbf 6$};
			\node()[above] at (1.5,0) {$\mbf 5$};
			\node()[above] at (2.5,0) {$\mbf 3$};
			\node()[above] at (3.5,0) {$\mbf 4$};
			\node()[above] at (4.5,0) {$\mbf 3$};
			\end{tikzpicture}
		\end{center}
	\end{example}
	
	As a consequence of Proposition \ref{prop:revCatalan}, we known that $\Slwk$ is nonempty if and only if $\kUnd$ is reverse Catalan. This can easily be shown combinatorially, as there is a tree $T$ that is in every nonempty slide set for a particular $n$, which we illustrate in the next example.
	
	\begin{example}
		The tree below lies in $\Slwk$ for all reverse-Catalan $\kUnd$. This is because the reverse-Catalan condition ensures that the edges are labeled in order from right to left.
		\begin{center}
			\includegraphics[scale=0.8]{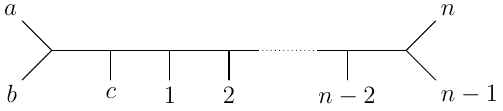}
		\end{center}
	\end{example}

	Another known fact about caterpillar trees is the characterization of the their edge labels in the case where $\kUnd=(1,1,\ldots,1)$, which we restate below. This result is Proposition 6.2 from \cite{GGL22}, which we restate here with modified notation.
	
	\begin{proposition}[Proposition 6.2 from \cite{GGL22}]
		Let $\Cat^\omega(1,1,\ldots,1)\subseteq\Slwl$ be the subset of trivalent trees that correspond to caterpillar curves. For each tree $T\in\Cat^\omega(1,1,\ldots,1)$, define the word $w(T)$ by reading the labels in the slide labeling of $T$ from left to right. The set of words $\{w(T) : T\in\Cat^\omega(1,1,\ldots,1)\}$ are precisely the $23{-}1$-avoiding permutations of length $n$, and in fact the words $w(T)$ are all distinct.
		\label{prop:catAvoid23-1}
	\end{proposition}
	
	We clarify the notion of pattern avoidance and the notation $23{-}1$ in the next subsection. In particular, see the definition of vincular patterns in Definition \ref{def:vincular}. 
	
	\subsection{Pattern Avoidance}
	\label{sec:patterns}
	
	We start by defining the main ideas of classical pattern avoidance, and then define two particular variants that will show up in this paper. For a more thorough summary of classical pattern avoidance, we refer the reader to \cite{bevanPPdefs}.
	
	\begin{definition}
		A permutation $\sigma$ \textbf{contains} another permutation $\tau$ if $\sigma$ contains a subword with the same relative order as $\tau$. Conversely, a permutation $\sigma$ \textbf{avoids} a pattern $\tau$ if $\sigma$ contains no subword with the same relative order as $\tau$. 
	\end{definition}
	
	Another way to formulate this concept is via \textit{reductions}.
	
	\begin{definition}
		The \textbf{reduction} $\red(w)$ of a permutation or word $w$ is the word obtained by replacing the $i$th smallest entry of $w$ by $i$ for all $i$. If $\red(w)=v$, we say $w$ \textbf{reduces} to $v$.
		\label{def:red}
	\end{definition}
	
	For example, $\red(2574)=1342$. Then, pattern containment can be rephrased as a question of containing a subword that reduces to the desired pattern.
	
	It is common practice to visualize words or patterns by graphing them. This is done by plotting the points $(i,\sigma_i)$ for $i=1,2,\ldots,n$ on an $n\times n$ grid. See the next example for an example of this.
	
	\begin{example}
		The word $14352$ contains the pattern $\tau=123$, since the subword $135$ has the same relative order as $\tau=123$. See the circled entries in the diagram below on the left. 
		
		On the other hand, the word $\sigma=35214$ avoids $\tau=123$, since none of its length 3 subwords are in increasing order. See the diagram below on the right.	
		\begin{center}
			\begin{tikzpicture}[scale=0.2, baseline=-\the\dimexpr\fontdimen22\textfont2\relax]
			\draw (0,0)--(7,0)--(7,7)--(0,7)--(0,0);
			\fill[black] (1.5,1.5) circle (0.25cm);
			\fill[black] (2.5,4.5) circle (0.25cm);
			\fill[black] (3.5,3.5) circle (0.25cm);
			\fill[black] (4.5,5.5) circle (0.25cm);
			\fill[black] (5.5,2.5) circle (0.25cm);
			\draw (1.5,1.5) circle (0.5cm);
			\draw (3.5,3.5) circle (0.5cm);
			\draw (4.5,5.5) circle (0.5cm);
			\end{tikzpicture}
			\hspace{2cm}
			\begin{tikzpicture}[scale=0.2, baseline=-\the\dimexpr\fontdimen22\textfont2\relax]
			\draw (0,0)--(7,0)--(7,7)--(0,7)--(0,0);
			\fill[black] (1.5,3.5) circle (0.25cm);
			\fill[black] (2.5,5.5) circle (0.25cm);
			\fill[black] (3.5,2.5) circle (0.25cm);
			\fill[black] (4.5,1.5) circle (0.25cm);
			\fill[black] (5.5,4.5) circle (0.25cm);
			\end{tikzpicture}
		\end{center}
	\end{example}
	
	\begin{definition}
		For a pattern $\tau$, we denote the set of permutations of length $n$ that avoid $\tau$ by $$\Av_n(\tau)=\{\sigma\in S_n: \sigma\text{ avoids }\tau\}.$$
		Similarly, for a collection of patterns $\tau^{(1)},\ldots,\tau^{(k)}$, we define the set of permutations of length $n$ that avoid each $\tau_i$ as $$\Av_n(\tau^{(1)},\ldots,\tau^{(k)})=\{\sigma\in S_n: \sigma\in\Av_n(\tau^{(i)})\,\forall i\in[k]\}.$$
	\end{definition}
	
	Throughout this work, we will primarily be working with words that are not permutations. We can extend all of these ideas to the case where both $\sigma$ and $\tau$ are words instead of permutations. 
	
	\begin{example}
		The word $24665347$ contains the pattern $1221$, since it contains the subword $4664$, which has the same relative order as $1221$. See the diagram below.
		\begin{center}
			\begin{tikzpicture}[scale=0.2, baseline=-\the\dimexpr\fontdimen22\textfont2\relax]
			\draw (0,0)--(10,0)--(10,9)--(0,9)--(0,0);
			\fill[black] (1.5,2.5) circle (0.25cm);
			\fill[black] (2.5,4.5) circle (0.25cm);
			\fill[black] (3.5,6.5) circle (0.25cm);
			\fill[black] (4.5,6.5) circle (0.25cm);
			\fill[black] (5.5,5.5) circle (0.25cm);
			\fill[black] (6.5,3.5) circle (0.25cm);
			\fill[black] (7.5,4.5) circle (0.25cm);
			\fill[black] (8.5,7.5) circle (0.25cm);
			\draw (2.5,4.5) circle (0.5cm);
			\draw (3.5,6.5) circle (0.5cm);
			\draw (4.5,6.5) circle (0.5cm);
			\draw (7.5,4.5) circle (0.5cm);
			\end{tikzpicture}
		\end{center}
	\end{example}
	
	Notationally, we denote the set of words with content $\kUnd$ that avoid a pattern $\tau$ by $\Av_{\kUnd}(\tau)$, and similarly for a collection of patterns.
	
	There are two further variants of pattern avoidance that we will define. The first is that of barred patterns. For a more comprehensive study of barred patterns, see \cite{barred_patterns}.
	
	\begin{definition}
		Let $\tau$ be a word where some letters are barred and others are unbarred. We call $\tau$ a \textbf{barred pattern}. We say that $\sigma$ \textbf{contains} $\tau$ if $\sigma$ has a subword with the same relative order as the non-barred letterns of $\tau$ that does \textit{not} extend to a subword with the same relative order as all of $\tau$. Otherwise, $\sigma$ \textbf{avoids} $\tau$. 
		\label{def:barred}
	\end{definition}
	
	\begin{example}
		The word $\sigma=231456$ contains the barred pattern $\tau=23\overline{1}$. Although the length-2 subword $23$ can extend to the subword 231, the subword $45$ (among others) cannot, meaning it is an instance of a $23\overline{1}$ pattern.
		
		On the other hand, the word $\sigma=234561$ avoids the barred pattern $\tau=23\overline{1}$, since any subword with relative order $23$ can be extended, by adding the letter $1$ at the end, to a subword with relative order $231$.
		\begin{center}
			\begin{tikzpicture}[scale=0.2, baseline=-\the\dimexpr\fontdimen22\textfont2\relax]
			\draw (0,0)--(8,0)--(8,8)--(0,8)--(0,0);
			\fill[black] (1.5,2.5) circle (0.25cm);
			\fill[black] (2.5,3.5) circle (0.25cm);
			\fill[black] (3.5,1.5) circle (0.25cm);
			\fill[black] (4.5,4.5) circle (0.25cm);
			\fill[black] (5.5,5.5) circle (0.25cm);
			\fill[black] (6.5,6.5) circle (0.25cm);
			\end{tikzpicture}
			\hspace{2cm}
			\begin{tikzpicture}[scale=0.2, baseline=-\the\dimexpr\fontdimen22\textfont2\relax]
			\draw (0,0)--(8,0)--(8,8)--(0,8)--(0,0);
			\fill[black] (1.5,2.5) circle (0.25cm);
			\fill[black] (2.5,3.5) circle (0.25cm);
			\fill[black] (3.5,4.5) circle (0.25cm);
			\fill[black] (4.5,5.5) circle (0.25cm);
			\fill[black] (5.5,6.5) circle (0.25cm);
			\fill[black] (6.5,1.5) circle (0.25cm);
			\end{tikzpicture}
		\end{center}
		
	\end{example}
	
	Finally, we can also impose adjacency conditions on some entries of our pattern. Such patterns are called \textit{vincular} patterns. Vincular patterns were first introduced in \cite{Babson2000} as generalized permutations, and a more detailed history on them can be found in \cite{Steingrimsson_2010}.
	
	\begin{definition}
		Let $\tau$ be a word of length $k$ with dashes between some entries. We call $\tau$ a \textbf{vincular pattern}. A word $\sigma$  \textbf{contains} $\tau$ if $\sigma$ contains a subword $\sigma'$ that has the same relative order as $\tau$, and for each $i\in[k-1]$, if the en entries $\tau_i,\tau_{i+1}$ in $\tau$ are not separated by a dash, then $\sigma'_i$ and $\sigma'_{i+1}$ come from adjacent entries $\sigma_j,\sigma_{j+1}$ in $\sigma$. Otherwise, we say $\sigma$ \textbf{avoids} the vincular pattern $\tau$.
		\label{def:vincular}
	\end{definition}
	
	In other words, to contain  a vincular pattern, we must have a consecutive subword with the relative order of the pattern, except that the dashes in the pattern indicate entries that need come from consecutive letters of our word.
	
	\begin{example}
		The word $\sigma=32541$ contains the pattern $\tau=23{-}1$, since it contains the subword $251$, where the 2 and 5 are adjacent, and 251 has the relative order 231.
		
		On the other hand, although $\sigma=43152$ contains the classical pattern $\tau=231$ (consider the subword $352$), $\sigma$ avoids the vincular pattern $\tau=23{-}1$, since the 3 and 5 in $352$ (as well as the 4 and 5 of 452) are not adjacent.
		\begin{center}
			\begin{tikzpicture}[scale=0.2, baseline=-\the\dimexpr\fontdimen22\textfont2\relax]
			\draw (0,0)--(7,0)--(7,7)--(0,7)--(0,0);
			\fill[black] (1.5,3.5) circle (0.25cm);
			\fill[black] (2.5,2.5) circle (0.25cm);
			\fill[black] (3.5,5.5) circle (0.25cm);
			\fill[black] (4.5,4.5) circle (0.25cm);
			\fill[black] (5.5,1.5) circle (0.25cm);
			
			\draw (2.5,2.5) circle (0.5cm);
			\draw (3.5,5.5) circle (0.5cm);
			\draw (5.5,1.5) circle (0.5cm);
			\end{tikzpicture}
			\hspace{2cm}
			\begin{tikzpicture}[scale=0.2, baseline=-\the\dimexpr\fontdimen22\textfont2\relax]
			\draw (0,0)--(7,0)--(7,7)--(0,7)--(0,0);
			\fill[black] (1.5,4.5) circle (0.25cm);
			\fill[black] (2.5,3.5) circle (0.25cm);
			\fill[black] (3.5,1.5) circle (0.25cm);
			\fill[black] (4.5,5.5) circle (0.25cm);
			\fill[black] (5.5,2.5) circle (0.25cm);
			\end{tikzpicture}
		\end{center}
	\end{example}
	
	\section{Preliminaries on combinatorics of slide trees}
	\label{sec:preliminaries}
	
	In this section we define several useful characteristics of slide trees and summarize notions that we will use throughout the rest of this work.
	
	\begin{definition}
		For $\kUnd$ a composition of $n$, define $\maxzero(\kUnd)$ to be the largest $z\in[n]$ such that $k_z=0$. If $\kUnd=(1,1,\ldots,1)$, set $\maxzero(1,1,\ldots,1)=c$ or $\maxzero(1,1,\ldots,1)=0$ as appropriate in the given context.
		Let $\kUnd$ be a composition of $n$. Define $\maxzero(\kUnd)$ to be the largest $j\in[n]$ such that $k_j=0$, if such a $j$ exists. Otherwise, define $\maxzero(\kUnd)=c$.
		\label{def:maxzero}
	\end{definition}
	
	Note that in \cite{CGM}, $\maxzero(\kUnd)$ was called $i_{\boldsymbol{k}}$.
	
	\begin{definition}
		For a reverse-Catalan composition $\kUnd$, define $\z(i):= \#\{j>i\,|\,k_j=0\}$.
		\label{def:z(i)}
	\end{definition}
	
	\begin{notation}
		In order to maintain clarity between edge and leaf labels, we use bolded labels for the edges of a tree and nonbolded labels for the leaves of a tree, like $\mbf x$ vs $x$. 
	\end{notation}
	
	Recall that in a trivalent tree, every vertex either has degree 1 (a leaf) or degree 3. We call an edge an \textbf{internal edge} if it connects two non-leaf vertices.     A leaf of a slide tree will always be adjacent to exactly one internal vertex. It is useful, however, to consider when two leaves are as close to each other in a tree as is possible. For a leaf $i$, call its unique neighbor $v_i$. We say two leaves $i$ and $j$ are \textbf{adjacent} if $v_i=v_j$. That is, $i$ and $j$ are adjacent (in the traditional sense) to the same internal vertex. For a leaf $i$ in a trivalent tree, one of the following two cases will always occur:
	\begin{itemize}
		\item The leaf $i$ is adjacent to another leaf $j$, and $v_i$ has one internal edge.
		\item  The leaf $i$ is adjacent to no other leaves and $v_i$ has two internal edges.
	\end{itemize}
	
	In a slide tree, the leaves $a$ and $b$ are always adjacent.  We treat the collection of $a$, $b$, $v_a=v_b$, and the two edges between them as the \textbf{root} of a tree. 
	
	\begin{notation}
		We standardize our drawings of trees by always drawing them with the root on the left, with the internal edge coming out of $v_a$ pointing to the right. This way, when we say something is \textit{left} of a leaf/edge/etc, we mean `towards the root', and when we say \textit{right}, we mean `away from the root'.
	\end{notation}
	
	\begin{definition}
		Let $T$ be a trivalent tree whose set of internal edges of $T$ form a path.
		We call $T$ a \textbf{caterpillar tree}.  We define the subset of slide trees that are also caterpillar trees by $\Cat^\omega(\kUnd)\subseteq\Slwk$ (and equivalently for $\Cat^\psi(\kUnd)$).
		\label{def:Caterpillar}
	\end{definition}
	
	With the root on the left, all the internal edges of a caterpillar tree are drawn horizontally. It is then natural to treat the edge labels of a caterpillar tree as a word, read off from left to right. 
	
	\begin{example}
		The tree below is a tree in $\Slide^\omega(0,0,1,2,1,2)$, drawn along with its edge labels from the slide algorithm. With the root drawn on the left, it is natural to read off the word 546643 from the edge labels.
		\begin{center}
			\includegraphics[scale=0.8]{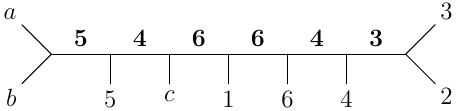}
		\end{center}
	\end{example}

	We say `the branch starting at edge $\mbf e$' to refer to the collection of all edges and vertices (including $\mbf e$ itself) on the opposite side of $\mbf e$ as the root.
	\begin{example}
		Consider the tree $T$ below on the left, and let $\mbf e$ be the bolded edge. Then, the branch starting at $\mbf e$ is the branch below on the right.
		\begin{center}
			\begin{tikzpicture}
			\draw (0,0)--(3,0);
			\draw[very thick] (1,0)--(2,0);
			\draw (1,0)--(1,-1);
			\draw (1,-1)--(0.5,-1.5);
			\draw (1,-1)--(1.5,-1.5);
			\draw (2,0)--(2,-0.6);
			\draw (3,0)--(3.5,0.5);
			\draw (3,0)--(3.5,-0.5);
			\draw (0,0)--(-0.5,0.5);
			\draw (0,0)--(-0.5,-0.5);
			
			\node[anchor=south east] at (-0.5,0.5) {$a$};
			\node[anchor=north east] at (-0.5,-0.5) {$b$};
			\node[anchor=north east] at (0.5,-1.5) {$c$};
			\node[anchor=north west] at (1.5,-1.5) {$1$};
			\node[anchor=north] at (2,-0.6) {$2$};
			\node[anchor=south west] at (3.5,0.5) {$4$};
			\node[anchor=north west] at (3.5,-0.5) {$3$};
			\end{tikzpicture}
			\hspace{2cm}
			\begin{tikzpicture}
			\draw (1,0)--(3,0);
			\draw[very thick] (1,0)--(2,0);
			\draw (2,0)--(2,-0.6);
			\draw (3,0)--(3.5,0.5);
			\draw (3,0)--(3.5,-0.5);
			
			\node[anchor=north] at (2,-0.6) {$2$};
			\node[anchor=south west] at (3.5,0.5) {$4$};
			\node[anchor=north west] at (3.5,-0.5) {$3$};
			\node[white,anchor=north east] at (0.5,-1.5) {$c$};
			\node[white,anchor=north west] at (1.5,-1.5) {$1$};
			\end{tikzpicture}
		\end{center}
		\label{ex:branch}
	\end{example}
	
	Similarly, we can consider the maximal branch that contains some leaf $j$ but not some other leaf $i$. For example, in Example \ref{ex:branch} above, the branch $B$ on the right is the maximal branch of $T$ that contains $3$, but not $1$, since adding any additional edges to $B$ would necessarily add the edge to the left of $\mbf e$, which would also add the leaves c and 1.
	
	\begin{definition}
		For a branch $B$ of a slide tree, let $\min(B)$ denote the minimal leaf label in $B$.
		\label{def:min}
	\end{definition}
	
	Next, we prove a number of basic facts about slide trees. We state these in terms of $\psi$-slide trees, but the statements are all also true for $\omega$-slide trees, as a result of the following lemma.
	
	\begin{lemma}
		Let $\kUnd$ be a composition of $n$. Then, $\Slwk\subseteq\Slpk$.
		\label{lem:omega<psi}
	\end{lemma}
	
	\begin{proof}
		This is an immediate consequence of the definition of the $\psi$ and $\omega$ $\kUnd$-slides. (See Definition \ref{def:k-slide}.)
	\end{proof}
	
	Although a slide tree, as defined, only has labels on its leaves, there is a unique edge label for each internal edge as given by the $\kUnd$-slide algorithm. Thus, throughout this work we will view a slide tree as consisting of the data of the underlying graph, along with both leaf and edge labels. 
	
	\begin{lemma}
		Let $i$ be a leaf of a tree $T\in\Slpk$. Then, all edges labeled $\mbf i$ must lie on the path from $i$ to $a$.
		\label{lem:pathtoa}
	\end{lemma}
	
	\begin{proof}
		This is immediate from Definition \ref{def:k-slide}.
	\end{proof}

	\begin{lemma}
		Let $v$ be a vertex of $T\in\Slpk$ that is adjacent to three internal vertices $\mbf x$, $\mbf y$, and $\mbf z$. Let $\mbf x$ be the edge towards $a$, and suppose (without loss of generality) that $\mbf{y}\geq\mbf{z}$. Then, the slide labeling algorithm  will always label the three edges in the order $\mbf y$, then $\mbf x$, and then $\mbf z$.
		\label{lem:slideorderyxz}
	\end{lemma}
	
	\begin{proof}
		By Lemma \ref{lem:pathtoa}, we cannot have $\mbf y=\mbf z$, so we must have that $\mbf y>\mbf z$. Thus, since the slide labeling algorithm uses labels in decreasing order, $\mbf y$ is labeled before $\mbf z$. For the rest, it suffices to show that $\mbf x$ is labeled second.
		
		We can never label $\mbf x$ first, since with neither $\mbf y$ and $\mbf z$ contracted, there is no leaf for the label $x$ to come from. Suppose instead that $\mbf x$ were labeled last. Let $B$ and $C$ be the branches starting at $\mbf y$ and $\mbf z$, respectively. When $\mbf y$ slides, since $\mbf x$ is still not contracted, the algorithm will compare $\min(B)$ against $\min(C)$, so we must have $\min(B)>\min(C)$. Similarly, if $\mbf x$ is still not contracted when $\mbf z$ slides, we will also require that $\min(C)>\min(B)$. These cannot possibly both be true at once, so we have a contradiction. Thus, $\mbf x$ also cannot be labeled last. 
	\end{proof}
	
	So, regardless of the exact labels involved, there is only one order we can label the three edges around an internal vertex. If we additionally consider the relative order of the three labels, we see that $\mbf z$ must always be a unique label, while $\mbf y$ can either be equal to or larger than $\mbf x$. The two options are shown in the figure below. As a consequence, we have the following corollary, which says that these are the only two types of internal vertices.
	
	\begin{center}
		\begin{tikzpicture}
		\draw (-.25,0)--(.4,0);
		\draw (1,0)--(3,0);
		\draw (2,0)--(2,-1);
		\draw (-.25,0)--(-0.75,0.5);
		\draw (-.25,0)--(-0.75,-0.5);
		
		\node[anchor=west] at (2,-0.5) {$1$};
		\node[anchor=south east] at (-0.75,0.5) {$a$};
		\node[anchor=north east] at (-0.75,-0.5) {$b$};
		\node[anchor=south] at (1.5,0) {$2$};
		\node[anchor=south] at (2.5,0) {$2$};
		\node at (.7,0) {$\ldots$};
		\node at (3.3,0) {$\ldots$};
		\node at (2,-1.15) {$\vdots$};
		\end{tikzpicture}
		\hspace{1.5cm}
		\begin{tikzpicture}
		\draw (-.25,0)--(.4,0);
		\draw (1,0)--(3,0);
		\draw (2,0)--(2,-1);
		\draw (-.25,0)--(-0.75,0.5);
		\draw (-.25,0)--(-0.75,-0.5);
		
		\node[anchor=west] at (2,-0.5) {$1$};
		\node[anchor=south east] at (-0.75,0.5) {$a$};
		\node[anchor=north east] at (-0.75,-0.5) {$b$};
		\node[anchor=south] at (1.5,0) {$2$};
		\node[anchor=south] at (2.5,0) {$3$};
		\node at (.7,0) {$\ldots$};
		\node at (3.3,0) {$\ldots$};
		\node at (2,-1.15) {$\vdots$};
		\end{tikzpicture}
	\end{center}
	
	\begin{corollary}
		Let $v$ be a vertex of $T\in\Slpk$ that is adjacent to three internal vertices $\mbf x$, $\mbf y$, and $\mbf z$. Let $\mbf x$ be the edge towards $a$, and suppose (without loss of generality) that $\mbf{y}\geq\mbf{z}$. Then, either $\mbf{x}=\mbf{y}>\mbf{z}$, or $\mbf{y}>\mbf{x}>\mbf{z}$.
		\label{cor:only212or231}
	\end{corollary}
	
	Intuitively, these two types of internal vertices correspond to $2{-}1{-}2$ patterns and $23{-}1$ patterns, respectively, appearing in the words corresponding to a tree.

	\begin{lemma}
		Let $v$ be a vertex of $T\in\Slpk$ that is adjacent to three internal vertices $\mbf x$, $\mbf y$, and $\mbf z$. Let $\mbf x$ the edge towards $a$, and $\mbf z$ the unique smallest edge given by Corollary \ref{cor:only212or231}. Let $B$ be the branch starting at $\mbf x$, and $B'$ the branch starting at $\mbf z$. Then, $\min(B)\in B'$.
		\label{lem:minleaf}
	\end{lemma}
	
	\begin{proof}
		By Lemma \ref{lem:slideorderyxz}, $\mbf y$ slides first out of $\mbf x$, $\mbf y$, and $\mbf z$. Let $C$ be the branch starting at $\mbf y$. When $\mbf y$ slides, it will compare then minimal element of $C$ to that of $B'$. So, we must have $\min(B')<\min(C)$, which must mean that $\min(B)=\min(B')$.
	\end{proof}
	
	\begin{remark}
		Given the edge labels of a slide tree, we can find the location of the leaf $c$. Starting at the root, walk down the internal edges until the labels weakly increase or we reach a branch. If we reach an ascent in the edge labels, the leaf between them is $c$, or else the larger edge (or the right one in the case of a tie) cannot slide. If we reach a branching vertex first, then by Lemma \ref{lem:minleaf} it must be down the branch with the smaller edge. Then, we can repeat this process with that branch until we reach a weak ascent across a leaf, or the end of a branch.
		\label{rmk:whereC}
	\end{remark}
	
	Our last result for this section is only for the case of $\omega$-slide trees.
	
	\begin{lemma}
		Let $i$ be a leaf of a tree $T\in\Slwk$ such that $v_i$ has two internal edges. Call these edges $\mbf x$ and $\mbf y$, with $\mbf x$ the edge towards $a$. If $i=\mbf x$, then $\mbf x>\mbf y$.
		\label{lem:descentleaf}
	\end{lemma}
	
	\begin{center}
		\begin{tikzpicture}
		\draw (-.25,0)--(.4,0);
		\draw (1,0)--(3,0);
		\draw (2,0)--(2,-0.6);
		\draw (-.25,0)--(-0.75,0.5);
		\draw (-.25,0)--(-0.75,-0.5);
		
		\node[anchor=north] at (2,-0.6) {$i$};
		\node[anchor=south east] at (-0.75,0.5) {$a$};
		\node[anchor=north east] at (-0.75,-0.5) {$b$};
		\node[anchor=south] at (1.5,0) {$\mbf i$};
		\node[anchor=south] at (2.5,0) {$\mathbf{< i}$};
		\node at (.7,0) {$\ldots$};
		\node at (3.3,0) {$\ldots$};
		\end{tikzpicture}
	\end{center}
	
	\begin{proof}
		We can not have $\mbf y=\mbf x=i$, since that would contradict Lemma \ref{lem:pathtoa}. So, we must show that if $\mbf x=i$, we can not have $\mbf y>\mbf x$. So, let $\mbf x=i$ and suppose instead that $\mbf y$ is larger than $\mbf x$. Then, $\mbf y$ slides before $\mbf x$ in the slide labeling. Let $B$ be the branch starting at $\mbf y$. When $\mbf y$ slides, since $\mbf x$ has not yet slid, $\mbf y$ will compare the minimal leaf of $B$ to $i$. For $\mbf y$ to be a valid slide, we must have $\min(B)>i$. However, then when $i$ later slides down edge $\mbf x$, it will need to compare $\min(B)$ with something further down the tree. Since $\min(B)>i$, this is an invalid slide, since it violates the definition of an $\omega$-slide in Definition \ref{def:k-slide}. Thus, for $T$ to be a valid $\omega$-slide tree, we must have $\mbf y<\mbf x$.
	\end{proof}

	\section{The main bijection}
	
	In this section, we now answer Question \ref{q:mainQuestion}. We do so by constructing an explicit bijection between $\Slwk$ and the set of CPF words. 
	
	Recall the \textit{asymmetric multinomial coefficients} $\multichoose{n}{\kUnd}$. They satisfy the recurrence relation given in Definition \ref{def:asymb}, which is discussed in more detail in \cite{CGM}. We now show combinatorially that the sets $\Slwk$ satisfy this same recursion.
	
	We will show this as follows. In \cite{CGM}, it was shown that the coefficient $\multichoose{n}{\kUnd}$ is equal to a sum of terms $\multichoose{n-1}{\kUnd^{(j)}}$ for compositions $\kUnd^{(j)}$ of $n-1$. We will define maps $\sigmaT_{i,j}$ and $\sigmaT_i$ from $\Slide^\omega(\kUnd^{(j)})$ to $\Slwk$, along with inverses $\piT_{i,j}$ and $\piT_i$. Then, we will show that these maps are injective and none of their images overlap, so the collection of these maps together defines a bijection between $\Slwk$ and the union of the sets $\Slide^\omega(\kUnd^{(j)})$.
	Then, we can find the word corresponding to a tree $T$ recursively as follows. Knowing that $T$ is in the image of exactly one map $\sigmaT_{\bullet}$, one can find which edge of $T$ was added by $\sigmaT_{\bullet}$, and find its label under the $\omega$ $\kUnd$-slide labeling algorithm. Then, the word for $T$ is that edge label appended to the end of the word for $\piT_{\bullet}(T)$.

	\subsection{Preliminaries}
	
	This technique relies on having a means of determining which image of a map $\sigmaT_\bullet$ a given tree $T$ is in. This is done by the map $\last(T)$, which we now define.
	
	\begin{definition}
		Let $T\in\Slwk$, and $B$ be a branch of $T$. Let $i$ and $j$ be the smallest and second smallest leaves of $B$, respectively.  Then, define $\mintwo(B)$ to be the largest branch of $B$ containing $j$ but not $i$.
	\end{definition}
	
	\begin{definition}
		Define the map $\last:\Slwk\to[n]$ as follows.
		\begin{enumerate}
			\item For a given $T\in\Slwk$, let $B$ be the largest branch of $T$ that has $\maxzero(\kUnd)$ as its smallest leaf.
			\item If $B$ has at least two leaves, replace $B$ with $\mintwo(B)$.
			\item Repeat Step 2 until $B$ has a single leaf.
			\item Define $\last(T)$ to be the single leaf of $B$.
		\end{enumerate}
		\label{def:last}
	\end{definition}
	
	\begin{remark}
		Note that since $T$ is a finite tree, this process terminates after a finite number of steps. Secondly, since the new $B$ must always contain the leaf $j$, every step gives a nonempty branch of $T$. Thus, we will always end up with a branch containing a single leaf, so the map $\last(T)$ is well-defined. 
	\end{remark}
	
	\begin{example}
		Consider the following tree $T\in\Slide^\omega(0,0,1,1,2,2)$.
		\begin{center}
			\includegraphics{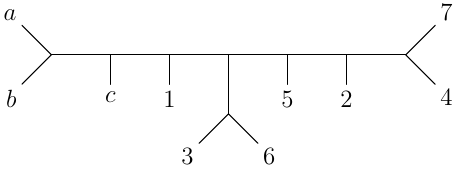}
		\end{center}
		We will compute $\last(T)$. First, $\maxzero(0,0,1,1,2,2)=2$, so our initial $B$ is the branch below on the left, since making it any larger would add the leaf $1<2$.
		\begin{center}
			\includegraphics{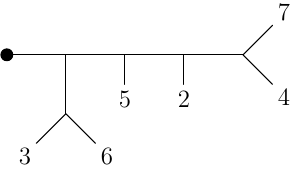}
			\hspace{3cm}
			\includegraphics{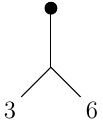}
		\end{center}
		Then, we apply $\mintwo$ to $B$ once to get the largest branch containing $3$ but not $2$, shown above on the right. Finally, applying $\mintwo$ a second time yields just the leaf $6$, and so $\last(T)= 6$.
	\end{example}
	
	\begin{lemma}
		For any $T\in\Slwk$, $\last(T)>\maxzero(\kUnd)$.
		\label{lem:larger_than_zero}
	\end{lemma}
	
	\begin{proof}
		Let $z=\maxzero(\kUnd)$. It is clear that $\last(T)\geq z$. We show that $\last(T)\neq z$. To do so, it is enough to demonstrate that the initial branch $B$ contains at least two leaves. 
		
		If $z$ is adjacent to another leaf, then since $z$ does not slide, the other leaf $x$ must label the adjacent edge, and for $x$ to slide we must have $z<x$. So, at a minimum, $B$ contains both of the leaves $z$ and $x$. Alternatively, suppose $z$ is between two edges $\mbf{x}$ and $\mbf{y}$, with $\mbf{x}$ on the path towards $a$. Since $z$ does not slide, $\mbf{x}\leq\mbf{y}$, so $\mbf{y}$ slides before $\mbf{x}$. For this to be a valid slide, $z$ must be smaller than everything on the branch starting at $\mbf{y}$, and so $B$ will include all of this branch. 
		
		Thus, the initial branch $B$ will contain at least two leaves. So, $\mintwo$ will be applied at least once, and thus $\last(T)>z$.
	\end{proof}
	
	\begin{lemma}
		The map $\last(T)$ returns a leaf $j$ adjacent to another leaf $i$ with $j>i$.
		\label{lem:largeleaf}
	\end{lemma}
	
	\begin{proof}
		We first show that $\last$ cannot return a leaf that is between two edges.  Suppose instead that $\last(T)=j$ is a leaf between two edges $\mbf{x}$ and $\mbf{y}$, where $\mbf{x}$ is the edge towards $a$. As shown in the previous proof, the initial branch $B$ has more than one leaf, so $\mintwo$ is applied at least once. Consider the state of $B$ just before the last time $\mintwo$ is applied. Since $B$ is a branch, and $B$ contains $j$ and at least one other leaf, $B$ must include the whole branch starting at $\mbf{x}$.
		Then, since $\mintwo(B)$ is just $j$, $j$ must be the second smallest leaf of $B$, and the smallest leaf $l$ must be in the branch starting at $\mbf{y}$. Since $\mbf{y}$ slides, we have $\mbf{y}>l$. By Lemma \ref{lem:descentleaf}, if $j=\mbf{x}$, then $j=\mbf{x}>\mbf{y}>l$, which contradicts that $j$ is the second smallest leaf of $B$. So, $j\neq\mbf{x}$.
		Then, $\mbf{x}\leq\mbf{y}$, so $\mbf{y}$ slides before $\mbf{x}$. However, $\mbf{y}$ cannot slide from $l$ to $j$ since $l<j$. Thus, we have a contradiction, so $\last(T)=j$ does not lie between two edges, and instead is adjacent to another leaf $i$.
		
		Finally, we must show that the leaf $i$ that is adjacent to $\last(T)=j$ is smaller than $j$. This follows immediately from the fact that the only way for $\mintwo$ to separate $j$ from $i$ is if for some branch, $i$ is the smallest leaf and $j$ is the second smallest leaf. Thus, $j>i$. 
	\end{proof}

	\subsection{The map $\sigmaT_{i,j}$}
	
	In the next two subsections, we define the two maps $\sigmaT_{i,j}$ and $\sigmaT_{j}$. 
	These maps take a slide tree $T$ and add an additional edge $j$ to create a larger slide tree. The map $\sigmaT_j$ corresponds to the case in Definition \ref{def:asymb} where we delete a 1 in position $j$ of $\kUnd$, and $\sigmaT_{i,j}$ corresponds to the case where we subtract 1 from an entry greater than 1 in position $j$, and delete the rightmost zero in position $i<j$. We will also show that $\last(\sigmaT_\bullet(T))$ will always return the label of the edge added to $T$ by $\sigmaT_\bullet$.
	
	\begin{definition}
		Let $\kUnd$ be a reverse-Catalan composition of $n$, $j$, and $i$ be integers such that $\maxzero(\kUnd)<i<j\leq n+1$, and $\kUnd'$ be the composition of $n+1$ obtained from $\kUnd$ by inserting a zero between $k_{i-1}$ and $k_i$, and then increasing the $j$th entry of the result by 1. Note that using the notation from Definition \ref{def:asymb}, $\kUnd=\kUnd'^{(j)}$. We define the map $\sigmaT_{i,j}: \Slwk\to\Slide^\omega(\kUnd')$ as follows.
		\begin{enumerate}
			\item Given a tree $T\in\Slwk$, add 1 to all leaf (and edge) labels greater than or equal to $i$.
			\item On the path from $a$ to $j$, consider the maximal length decreasing sequences of edge labels. 
			\item Let $B_1,B_2,\ldots,B_l$ be the branches of $T$ off of this path that lie between these maximal length decreasing sequences, and $B_l$ the branch immediately next to leaf $j$. (See Figure \ref{fig:sigmaEx} for an example.) Note that some (or all) of these branches may consist of a single leaf, and there may be additional branches that connect to this path in the middle of a decreasing sequence.
			\item For $r\in[l]$, let $m_r:=\min(B_r)$.
			\item As we will show in Lemma \ref{lem:3cases}, we need only consider the following three cases for the ordering of $m_1,\ldots,m_l$, $i$, and $j$. For each case, we say how to get $\sigmaT_{i,j}(T)$:
			\begin{itemize}[topsep=4pt,itemsep=0pt]
				\item $m_l<i<j$ \textit{or} $m_{l-1}<i<j<m_l$: Replace the leaf $j$ by an edge labeled $\mbf{j}$ with leaves $j$ and $i$.
				\item $m_1<\cdots<m_{d-1}<i<m_d<\cdots<m_l<j$: Replace the leaf $m_d$ by $i$, $m_{d+1}$ by $m_d$, and so on, replace $m_l$ by $m_{l-1}$, and replace leaf $j$ by an edge $\mbf{j}$ with leaves $j$ and $m_l$.
				\item $m_1<\cdots<m_{d-1}<i<m_d<\cdots<j<m_l$: Replace the leaf $m_d$ by $i$, $m_{d+1}$ by $m_d$, and so on up to replacing $m_{l-1}$ with $m_{l-2}$, then replace leaf $j$ by an edge $\mbf{j}$ with leaves $j$ and $m_{l-1}$.
			\end{itemize}
			Note that the first case is actually subsumed by the second two cases, but we write it out separately for clarity and to make the proofs clearer.
		\end{enumerate}
		\label{def:sigmaIJ}
	\end{definition}
	
	\begin{remark}
		In this definition, and as necessary throughout this section, we use the abuse of notation of using the same label to refer to a tree $T\in\Slwk$ as to the tree after having some of its leaf and edge labels changed in Step 1 of Definition \ref{def:sigmaIJ}.
	\end{remark}
	
	\begin{figure}[bt]
		\centering
		\includegraphics{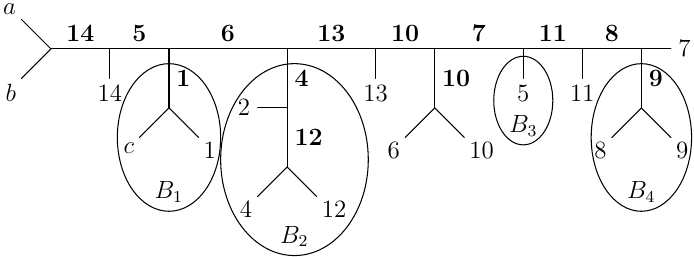}
		\caption{A slide tree in $\Slide^\omega(1,0,1,1,1,1,1,1,2,1,1,1)$ with the labels $\geq3$ incremented by 1. In this case, $j=7$.}
		\label{fig:sigmaEx}
	\end{figure}
	
	\begin{example}
		Consider the tree $T\in\Slide^\omega(1,0,1,1,1,1,1,1,2,1,1,1)$ depicted in Figure \ref{fig:sigmaEx}. We will demonstrate how we construct the tree $\sigmaT_{3,7}(T)$.
		The figure depicts $T$ after we have already incremented all the leaf and edge labels greater than or equal to 3 by 1. Then, the maximal decreasing sequences of edge labels from $a$ to $7$ are $(14,5)$, $(6)$, $(13,10,7)$, and $(11,8)$. Then, $l=4$ and the branches $B_1,B_2,B_3, B_4$ are as depicted in the figure. Their minimal leaves are $m_1=c$, $m_2=2$, $m_3=5$, and $m_4=8$. We have $c<2<i=3<5<i=7<8$, so $d=3$ and we are in the third case of Step 5. Thus, we replace $5$ with $3$ and $7$ with an edge $\mbf{7}$ with leaves $7$ and $5$ to form $\sigmaT_{3,7}(T)\in \Slide^\omega(1,0,0,1,1,1,2,1,1,2,1,1,1)$:
		\begin{center}
			\includegraphics{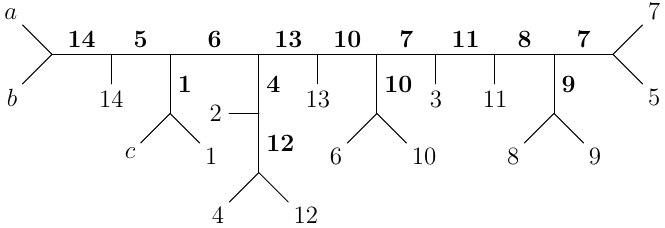}
		\end{center}
	\end{example}
	
	To show that our map $\sigmaT_{i,j}$ is well defined, we first must demonstrate a few facts about the leaves $m_1,\ldots,m_l$, which we do in the next two lemmas.

	\begin{lemma}
		Let $j$ be a leaf of $T\in\Slwk$, and let $B_1,\ldots,B_l$ and $m_1,\ldots,m_l$ be as in Definition \ref{def:sigmaIJ}. Then,
		\begin{enumerate}[(1),noitemsep, nolistsep]
			\item $m_1=c$,
			\item $m_1<m_2<\cdots<m_l$, and
			\item $m_{l-1}<j$.
		\end{enumerate}
		\label{lem:3cases}
	\end{lemma}
	
	\begin{proof}
		For part (1), consider the path from $a$ to $B_1$. Since the sequence of edge labels decreases, any leaves here must be the same as the adjacent edge label. For any internal vertex $v$ between $a$ and $B_1$, the edge going towards $B_1$ must be the smallest edge adjacent to $v$, so $c$ cannot be down any branches that split off here by Remark \ref{rmk:whereC}.
		Then, where $B_1$ splits off, if $B_1$ has any edges, the edge at the base of $B_1$ must be smaller than the two edges on the path from $a$ to $j$ by Lemma \ref{lem:slideorderyxz}, 
		so it must contain $c$ by Remark \ref{rmk:whereC}. If $B_1$ is just a leaf, then it must be $c$ anyway, or else the edge to its right (which slides before the one to the left) is unable to slide. This proves Lemma \ref{lem:3cases}(1).
		
		For claim (2), suppose instead that there is some $d$ for which $m_d>m_{d+1}$. Then, the edge on the path from $a$ to $j$ immediately to the right of $B_d$ cannot slide, as it would be comparing something to its right that is at most $m_{d+1}$ to $m_d$, but we assumed that $m_d>m_{d+1}$. Thus, we must instead have that $m_1<\cdots<m_l$.
		
		Finally, for (3), the proof follows identically to that of (2): if $m_{l-1}>j$, then the edge to the right of $B_{l-1}$ cannot make a valid slide, so we do not have a valid slide tree, and we must have that $m_{l-1}<j$.
	\end{proof}
	
	\begin{corollary}
		As a consequence, we must have either $m_l<i<j$, $m_{l-1}<i<j<m_l$, $m_1<\cdots<m_{d-1}<i<m_d<\cdots<m_l<j$, or $m_1<\cdots<m_{d-1}<i<m_d<\cdots<j<m_l$. Thus, the three cases in Step 5 of Definition \ref{def:sigmaIJ} are the only three cases to consider.
		\label{cor:3cases}
	\end{corollary}
	
	\begin{lemma}
		Let $j$ be a leaf of $T\in\Slwk$, let $B_1,\ldots,B_l$ and $m_1,\ldots,m_l$ be as in Definition \ref{def:sigmaIJ}, and let $s\in\{2,3,\ldots,l\}$. If $s<l$, or $s=l$ and $j>m_l$, then $m_s$ is the smallest leaf on the maximal branch containing $m_s$ but not $m_{s-1}$.
		\label{lem:m_sSmallest}
	\end{lemma}
	
	\begin{proof}
		Let $s\in\{2,3,\ldots,l\}$ as above, and let $x$ be the smallest leaf right of $B_s$. We first show that any leaf off of the path from $B_{s-1}$ to $B_s$ is larger than $\min(m_s,x)$.
		Let $\mbf{e_1}>\cdots>\mbf{e_{r+1}}$ be the edges from $B_{s-1}$ to $B_{s}$, and $E_1,\ldots,E_r$ the branches between them. If $E_r$ is a leaf, then that leaf must be labeled $e_r$ since $\mbf{e_r}>\mbf{e_{r+1}}$. Otherwise, consider the edge $\mbf f$ of $E_r$ adjacent to $\mbf{e_r}$ and $\mbf{e_{r+1}}$. Since $\mbf{e_r}>\mbf{e_{r+1}}$, $\mbf f\geq \mbf{e_r}$ by Lemma \ref{lem:slideorderyxz}. 
		So, $\mbf f$ slides before $\mbf{e_r}$, and $\min(E_r)$ must be larger than the smallest leaf to the right of $E_r$. Similarly, we can show that $$\min(E_{r-1})>\min(m_s,x,\min(E_r))=\min(m_s,x)$$ by the same argument. So, continuing this process, we have that every leaf off of the path between $B_{s-1}$ and $B_s$ is larger than $\min(m_s,x)$.
		
		Next, we proceed by downwards induction on $s$. We have two base cases to consider: either $s=l$ when $j>m_l$, or $s=l-1$ when $j<m_l$.  In the first case, the only leaf right of $B_l$ is $j$, and $j>m_l$, so $m_l$ is the smallest leaf in or right of $B_s$. So, by the previous argument, $m_l$ is the smallest leaf on the maximal branch containing $m_l$ but not $m_{l-1}$.  For the second case, by the above argument everything between $B_{l-1}$ and $B_l$ is larger than $j=\min(m_l,j)$. By Lemma \ref{lem:3cases}(3), $m_{l-1}<j$, so $m_{l-1}$ is the smallest leaf in or right of $B_{l-1}$, so by the previous paragraph $m_{l-1}$ is the smallest leaf in the maximal branch containing $m_{l-1}$ but not $m_{l-2}$. 
		
		By induction, we assume this is true for $m_{s+1}$. Then, by an identical argument to $m_{l-1}$, using the fact that $m_1<m_2<\cdots<m_l$ by Lemma \ref{lem:3cases}(2) we can show the same is true for $m_s$.
	\end{proof}
	
	\begin{theorem}
		The map $\sigmaT_{i,j}$ is well-defined.
		\label{thm:sigmaIJwelldef}
	\end{theorem}
	
	\begin{proof}
		The algorithm described in Definition \ref{def:sigmaIJ} clearly gives a unique output, so showing well-definedness amounts to showing that the result is an element of $\Slide^\omega(\kUnd')$. We do this by first showing that the leaves in $T'=\sigmaT_{i,j}(T)$ will label the same edges as in $T$ (except for the new edge that was created), and then show that every slide is valid.
		
		The only leaves in different positions between $T$ and $T'$ (other than $i$, which does not slide at all) are $m_d,\ldots,m_l$ (or $m_d,\ldots,m_{l-1}$ when $j<m_l$). However, by Lemma \ref{lem:m_sSmallest}, each such $m_s$ is the smallest leaf in a branch containing both $m_s$ and the position $m_s$ is moved to by $\sigmaT_{i,j}$. So, every edge in that branch is labeled by something larger than $m_s$, and moving $m_s$ within that branch does not change which edge it will end up labeling. Thus, assuming that every slide is valid, the edges of $T'$ will be labeled in the same order (and thus get the same labels) as in $T$. 
		
		Next, we show that every edge $\mbf e$ in $T'$ is a valid slide, given that $T$ is a valid slide tree. We do this by case of where $\mbf e $ is in $T$ in relation to $P$, where $P$ is the path from $a$ to $j$. 
		
		\textbf{Case 1:} Suppose $\mbf e$ is on the path $P$. In $T$, $\mbf e$ slides from $m_{r+1}$ to $m_{r}$ for some $r<l$, by Lemma \ref{lem:m_sSmallest} and the fact that the edges on $P$ decrease from $B_r$ to $B_{r+1}$. Let $v_1$ and $v_2$ be the leaves labeled by $m_r$ and $m_{r+1}$ in $T$, respectively. Although $\sigmaT_{i,j}$ might relabel one or both of $v_1$ and $v_2$, in all cases of Definition \ref{def:sigmaIJ}, $v_1$ will still have a smaller label than $v_2$, and, so long as $r\neq l{-}1$ or $j>m_l$, both leaves will still have the smallest labels in their respective branches mentioned in Lemma \ref{lem:m_sSmallest}. So in $T'$, $\mbf e$ will slide from $v_2$ to $v_1$, which is a valid slide.
		
		We also must consider the case where $\mbf e$ is to the right of $B_{l-1}$ and $j<m_l$. By Lemma \ref{lem:descentleaf}, $j$ is strictly smaller than the first edge to the left of $B_l$ on $P$. 
		So, $\mbf e$ slides from $j$ to $m_{l-1}$ in $T$. As $m_{l-2}<m_{l-1} <j<m_l$, $\mbf e$ will slide from $m_{l-1}$ to $m_{l-2}$ in $T'$, and this is a valid slide.
		
		\textbf{Case 2:} Suppose $\mbf e$ slides to $P$, and we are not in the case where $\mbf e$ is in $B_l$ and $j<m_l$. Let $s$ be the largest index such that $B_s$ is to the left of $\mbf e$. Then $\mbf e$ is either on $B_{s+1}$ or on a branch between $B_s$ and $B_{s+1}$.
		
		\textbf{Subcase 2(a):} Consider when $\mbf e$ is on a branch $B'$ between $B_s$ and $B_{s+1}$. Then, the larger minimal leaf used in the slide labeling algorithm is on $B'$, and so is some not $m_r$. So, the smaller leaf used in the comparison of the slide algorithm is either unchanged or made smaller, and the larger leaf is unchanged. So, since $\mbf e$ is a valid slide in $T$, $\mbf e$ is a valid slide in $T'$.
		
		\textbf{Subcase 2(b):} Consider when $\mbf e$ is on $B_{s+1}$. For any $\mbf f$ on $P$ between $B_s$ and $B_{s+1}$, $\mbf f>\mbf e$, since the edges of $P$ decrease from $B_s$ to $B_{s+1}$. So, $\mbf e$ slides in $T$ to some $m_r$ with $r\leq s$. If $\mbf e$ slides in $B_{s+1}$ from a leaf other than $m_{s+1}$, then the larger label in the slide algorithm comparison made by $\mbf e$ is unchanged by $\sigmaT_{i,j}$. The smaller label, $m_r$, is either unchanged or is replaced by something smaller, so $\mbf e$ is still a valid slide. 
		Otherwise, $\mbf e$ slides from $m_{s+1}$. If $\sigmaT_{i,j}$ does not change $m_{s+1}$, then $\mbf e$ is still valid by the previous argument. If $m_{s+1}$ is changed, it is replaced with either $i$ or $m_s$. If $m_{s+1}$ is replaced by $i$, then $\mbf e$ will slide from $i$ to $m_r$. Since $i>m_r$, $\mbf e$ is still a valid slide. If $m_{s+1}$ is replaced by $m_s$, then either $m_r$ in unchanged and is smaller than $m_s$, or $m_r$ is replaced by something smaller than itself. In either case, $\mbf e$ can slide from $m_s$ or $i$ to the label in the position of $m_r$ in $T$. Thus, $\mbf e$ is a valid slide. 
		
		\textbf{Case 3:} Suppose $\mbf e$ slides to $P$, is in $B_l$, and $j<m_l$. Then, since $B_l$ is unchanged by $\sigmaT_{i,j}$ in this case, the larger leaf used by $\mbf e$ in the slide algorithm is unchanged. If the smaller leaf label is changed (such as changing from $j$ to the new leaf $m_{l-1}$ added next to $j$), it is changed to something smaller. So, the comparison is still valid, and $\mbf e$ is a valid slide.
		
		\textbf{Case 4:} Suppose $\mbf e$ does not slide to $P$. Let $B$ be the branch off of $P$ containing $\mbf e$. Then in $T$, $\mbf e$ slides from some $v_1$ within $B$ to some $v_2$ also within $B$. If $\sigmaT_{i,j}$ changes $v_2$, it will replace $v_2$ with a smaller leaf label. Meanwhile, $\sigmaT_{i,j}$ does not change $v_1$. So, in $T'$ $\mbf e$ will slide from $v_1$ to either $v_2$ or something smaller, so $\mbf e$ is still a valid slide.
		
		Finally, we consider the new edge $\mbf j$. In Case 1 of Step 5, $j$ slides from $i$ to either $m_l$ or $m_{l-1}$. In Case 2, $j$ slides from $m_l$ to $m_{l-1}$, and in Case 3, $j$ slides from $m_{i-1}$ to $m_{i-2}$. In all three cases, the leaf slid from is larger than the leaf slide to, so the edge $\mbf j$ is a valid slide. Thus, every edge of $T'$ is a valid slide, and so $T'$ is an element of $\Slide^\omega(\kUnd')$, meaning that $\sigmaT_{i,j}$ is well-defined.
	\end{proof}

	\begin{lemma}
		For $\kUnd$ a composition of $n$ and $T\in\Slwk$, $\last(\sigmaT_{i,j}(T))=j$.
		\label{lem:lastreturnslastIJ}
	\end{lemma}
	
	\begin{proof}
		If $T'=\sigmaT_{i,j}(T)$, then $\maxzero(\kUnd')=i$. Let $l$, $d$, $B_1,\ldots,B_l$, and $m_1,\ldots,m_l$ be as in Definition \ref{def:sigmaIJ}. We first show that the maximal branch of $T'$ containing $i$ as its smallest leaf (i.e. the branch $B$ in Step 1 of Definition \ref{def:last}) is the maximal branch not containing $B_{d-1}$. Clearly it can be no larger than this, as $i>m_{d-1}$. By Lemma \ref{lem:m_sSmallest}, $m_d$ is the smallest leaf in the corresponding branch in $T$. So, since $i<m_d$, $i$ is the smallest leaf on this branch in $T'$.
		
		Next, we iteratively apply $\mintwo$ to $B$. Since $m_d$ is smaller than every leaf of $B$ except $i$, $\mintwo(B)$ is the largest branch of $T'$ containing $m_d$ but not $i$. Then, since for $s\geq d$, the leaf $m_s$ in $T'$ is labeled $m_{s+1}$ in $T$, the same process will repeat, until we are left with the maximal branch with minimal leaf either $m_l$ (in Case 2 of Step 5 above) or $m_{l-1}$ (in Case 3 of Step 5 above). 
		If we are in Case 2, then $j$ and $m_l$ are the only two leaves left in $B$, so applying $\mintwo$ one last time leaves just the leaf $j$. In Case 3, applying $\mintwo$ one more time leaves us with just $j$, since $j<m_l$. Note that we do not include Case 1 here, since it is covered by Cases 2 and 3. In any case, we get that $\last(\sigmaT_{i,j}(T))=j$.
	\end{proof}
	
	Next, we show that $\sigmaT_{i,j}$ is injective by building an inverse map $\piT_{i,j}$ that undoes $\sigmaT_{i,j}$. From the definition of $\sigmaT_{i,j}$, in $\sigmaT_{i,j}(T)$ the leaf $j$ is adjacent to some other leaf $v$ and an edge labeled $\mbf j$. Secondly, $i$ is the largest leaf label of $\sigmaT_{i,j}(T)$ that does not slide. With that in mind:
	\begin{definition}
		Define $$\piT_{i,j}:\sigmaT_{i,j}(\Slwk)\to\Slwk$$ as follows. For $T\in\sigmaT_{i,j}(\Slwk)$, $\last(\sigmaT_{i,j}(T))=j$ by Lemma \ref{lem:lastreturnslastIJ}. Let $v$ be the leaf adjacent to leaf $j$. 
		
		If $v=i$, then replace the branch consisting of the edge $\mbf j$ and leaves $j$ and $i$ with the leaf $j$. Then, subtract 1 from all labels in $T$ greater than $j$. Define $\piT_{i,j}(T)$ to be the result.
		
		Otherwise, $v\neq i$, so define branches $B_1,\ldots,B_l$ and leaves $m_1,\ldots,m_l$ as in Definition \ref{def:sigmaIJ}. (Note: Since $\sigmaT_{i,j}$ does not change any edges, but does shuffle the leaf labels, the branches $B_1,\ldots,B_l$ are the same in $T$ and $\sigmaT_{i,j}^{-1}(T)$, but the leaves $m_1,\ldots,m_l$ may be different.) By definition of $\sigmaT_{i,j}$, there is some leaf $m_d$ such that $m_d=i$. If $m_l<j$, replace leaf $m_l$ with $v$ and $m_{l-1}$ with $m_l$. Otherwise, replace $m_{l-1}$ with $v$. Then, regardless, replace $m_{l-2}$ with $m_{l-1}$, and so on, until we replace $m_d=i$ with $m_{d+1}$. Next, replace the branch consisting of the edge $\mbf j$ and leaves $j$ and $v$ with a leaf $j$. Then, subtract 1 from all labels in $T$ greater than $i$. Define $\piT_{i,j}(T)$ to be the result.
	\end{definition}
	
	\begin{lemma}
		For $T\in\Slwk$, $\piT_{i,j}(\sigmaT_{i,j}(T))=T$.
		\label{lem:pisigmaIJ}
	\end{lemma}
	
	\begin{proof}
		It is clear from the construction of $\sigmaT_{i,j}$ and $\piT_{i,j}$ that $\piT_{i,j}$ undoes the changes made by $\sigmaT_{i,j}$ to $T$.
	\end{proof}
	
	As an immediate consequence, we get the following corollary.

	\begin{corollary}
		The map $\sigmaT_{i,j}$ is injective.
		\label{cor:ijInj}
	\end{corollary}
	
	\subsection{The map $\sigmaT_{j}$}
	
	In this subsection, we define the other map from $\Slwk$ into $\Slide^\omega(\kUnd')$, along with its inverse. This one corresponds to the case in the asymmetric multinomial recursion where we decrement a 1 in position $j$, thus removing the 0 that then appears in that position. Hence, this map uses only one subscript.
	
	\begin{definition}
		Let $\kUnd$ be a reverse-Catalan composition of $n$, $j$ an integer such that $\maxzero(\kUnd)<j\leq n+1$, and $\kUnd'$ be the composition of $n+1$ obtained from $\kUnd$ by inserting a 1 between $k_{j-1}$ and $k_j$. We define the map $\sigmaT_{j}: \Slwk\to\Slide^\omega(\kUnd')$ as follows.
		\begin{enumerate}
			\item Given a tree $T\in\Slwk$, add 1 to all leaf (and edge) labels greater than or equal to $j$.
			\item Consider the leaf $v=\last(T)$. By Lemma \ref{lem:largeleaf}, it is at the end of an edge along with some other leaf $i<v$. There are three cases to  consider:
			\begin{itemize}[topsep=4pt,itemsep=0pt]
				\item $v<j$: Replace the leaf $v$ by an edge $\mbf{j}$ with leaves $j$ and $v$. The result is $\sigmaT_j(T)$.
				\item $i<j<v$: Replace the leaf $i$ by an edge $\mbf{j}$ with leaves $j$ and $i$. The result is $\sigmaT_j(T)$.
				\item $j<i$: Continue on to Step 3.
			\end{itemize}
			\item On the path from $a$ to $v$, consider the maximal length decreasing sequences of edge labels. 
			\item Let $B_1,B_2,\ldots,B_l$ be the branches away from this path between the maximal decreasing sequences, with $B_l$ the branch immediately next to leaf $v$. (See Example \ref{ex:sigmaJ}.)
			\item For $s\in[l]$, let $m_s:=\min(B_s)$.
			\item As already shown in Lemma \ref{lem:3cases}, $c=m_1<m_2<\cdots<m_l$. So, there is some $d$ such that $m_{d}<j<m_{d+1}$.
			\item Replace the leaf $m_d$ by an edge $\mbf{j}$ with leaves $j$ and $m_d$. The result is $\sigmaT_{j}(T)$.
		\end{enumerate}
		\label{def:sigmaJ}
	\end{definition}
	
	\begin{example}
		Consider the tree $T\in\Slide^\omega(1,0,1,1,1,2,1,1)$ below on the left. We will compute $\sigmaT_3(T)$.
		\begin{center}
			\includegraphics[scale=0.8]{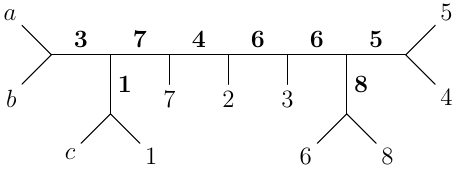}
			\hspace{1cm}
			\includegraphics[scale=0.8]{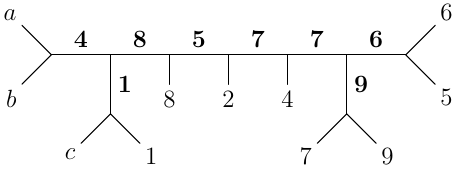}
		\end{center}
		The first step is to increment every label greater than or equal to 3, which gives us the tree above on the right.
		
		Then, $\last(T)=5$, so on the relabeled tree $v=6$ and $i=5$. Since $3<5<6$, we are in the third case, so we consider the maximal length decreasing subsequences from $a$ to $6$, which are $(4)$, $(8,5)$, $(7)$, and $(7,6)$. So, the branches $B_1,\ldots,B_4$ are as shown below.
		
		\begin{center}
			\includegraphics{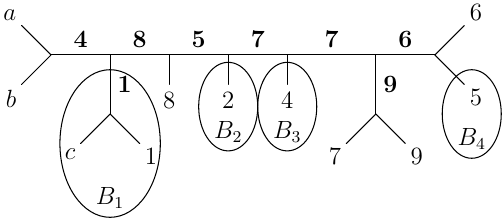}
		\end{center}
		
		So, $m_1=c$, $m_2=2$, $m_3=4$, and $m_4=5$. Since $2<3<4$, we replace the leaf 2. Thus, we get the tree $\sigmaT_3(T)$ below.
		
		\begin{center}
			\includegraphics{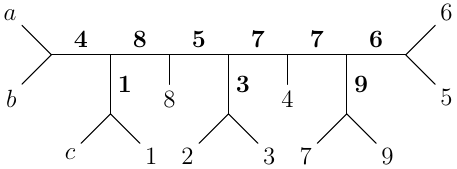}
		\end{center}
		\label{ex:sigmaJ}
	\end{example}
	
	\begin{theorem}
		The map $\sigmaT_j$ is well-defined.
	\end{theorem}
	
	\begin{proof}
		To show that $\sigmaT_j$ is well-defined, we need to show that the result $T'$ is an element of $\Slide^\omega(\kUnd')$. We consider the three separate cases of Step 2 in Definition \ref{def:sigmaJ}. For all three cases, the only change done to $T$ is replacing one leaf $i$ with an edge $\mbf j$ and leaves $j$ and $i$, where $j>i$. So, this does not affect the ability for any other label to slide, and we need only check that $\mbf j$ is a valid slide.
		
		In the first case, $j>l>i$, so $j$ can successfully slide from $l$ to $i$. In the second case, $l$ slides from $i$ to some $d<i$ in $T$. As $j<l$ slides after $l$, $j$ will compare $i$ to something that is at most $d$, so $j$ is a valid slide. 
		
		Finally, in the third case, $\maxzero(\kUnd)<j<m_{d+1}$, so when we found $\last(T)$, the branch $B$ in Step 1 included all of $B_d,\ldots,B_l$. So, $m_{d+1}$ must be smaller than every leaf in $B_d$ except $m_d$, or else $\last(T)$ would be a leaf in $B_d$. Thus, since $j<m_{d+1}$ (by Step 6 of Definition \ref{def:sigmaJ}), every edge in $B_d$ slides before $j$, so $j$ will compare $m_d$ to some leaf further down the tree, that can be at most $m_{d-1}$. Therefore, $j$ is a valid slide, and in all cases $\sigmaT_j(T)$ is a valid slide tree.
	\end{proof}
	
	\begin{lemma}
		For $\kUnd$ a composition of $n$ and $T\in\Slwk$, $\last(\sigmaT_{j}(T))=j.$
		\label{lem:lastreturnslastI}
	\end{lemma}
	
	\begin{proof}
		Let $i$ and $v$ be as in Definition \ref{def:sigmaJ}. We first consider the first two cases in Step 2. We are given that $\last(T)=v$. So, there is some branch $B$ with smallest leaf $i$ and second smallest leaf $v$ that is obtained at some step in computing $\last(T)$. Adding a leaf $j>i$ to this branch (as we do in these two cases) does not change any steps up to this point in computing $\last(\sigmaT_j(T))$. So, in the first case, applying $\mintwo$ to this branch $B$ gives the branch with only the leaves $j$ and $v$, since $j>v$. Then, applying $\mintwo$ a second time leaves us with just $j$. Alternatively, in the second case, $i<j<v$, so applying $\mintwo$ to $B$ separates $j$ from $i$, and gives just the leaf $j$.
		
		Next, we consider the third case. As in the proof of the previous lemma, the initial branch $B$ in Step 1 of computing $\last(T)$ includes all of $B_d,\ldots,B_l$. So, at some point during that process we have a $B$ with smallest leaf $m_d$ and second smallest leaf $m_{d+1}$. Adding an edge and leaf $j$ to $m_d$ does not change any of the steps up to this point when computing $\last(\sigmaT_j(T))$, but $j<m_{d+1}$, so $j$ is the new second smallest leaf of this branch. So, applying $\mintwo$ to this $B$ gives just the leaf $j$. Thus, in each case $\last(\sigmaT_{j}(T))=j$.
	\end{proof}
	
	We will now show that $\sigmaT_j$ is injective by building an inverse map $\piT_j$ that undoes $\sigmaT_j$. Notice that in all three cases of Definition \ref{def:sigmaJ}, all we do to $T$ is replace some leaf $i$ by an edge $\mbf j$ with leaves $j$ and $i$, where $j$ occurs as exactly one edge label in $\sigmaT_j(T)$. With that in mind:
	
	\begin{definition}
		Define $$\piT_j:\sigmaT_j(\Slwk)\to\Slwk$$ as follows. For $T\in\sigmaT_j(\Slwk)$, $\last(\sigmaT_{j}(T))=j$ by Lemma \ref{lem:lastreturnslastI}. There is exactly one edge in $T$ with label $\mbf j$, and it has leaves $j$ and $l$. Replace the branch consisting of this edge and these two leaves by the leaf $l$. Then, subtract 1 from all labels in $T$ greater than $j$. Let the resulting tree be $\piT_j(T)$.
	\end{definition}
	
	\begin{lemma}
		For $T\in\Slwk$, $\piT_j(\sigmaT_j(T))=T$.
		\label{lem:pisigmaJ}
	\end{lemma}
	
	\begin{proof}
		It is clear from the construction of $\sigmaT_j$ and $\piT_j$ that $\piT_j$ undoes the changes made to $T$ by $\sigmaT_j$.
	\end{proof}
	
	\begin{corollary}
		The map $\sigmaT_j$ is injective.
		\label{cor:jInj}
	\end{corollary}
	
	\begin{proof}
		This follows immediately from \ref{lem:pisigmaJ}.
	\end{proof}
	
	\subsection{Constructing the full bijection}
	
	So far, we have defined two maps $\sigmaT_{i,j}$ and $\sigmaT_{j}$, and shown that for any tree in their image, the function $\last()$ returns the value of the added edge. That is:
	
	\begin{theorem}
		Let $T\in \Slwk$. If $T\in\sigmaT_{i,j}(\Slide^\omega(\kUnd'))$ or $T\in\sigmaT_{j}(\Slide^\omega(\kUnd'))$ for some composition $\kUnd'$, then $\last(T)=j$.
		\label{thm:lastadded}
	\end{theorem}
	
	\begin{proof}
		This follows immediately from Lemmas \ref{lem:lastreturnslastIJ} and \ref{lem:lastreturnslastI}.
	\end{proof}
	
	Our ultimate goal is to combine these maps into a bijection from a disjoint union of slide sets for compositions of $n-1$ to $\Slwk$. We will next define this desired disjoint union, recalling Definition \ref{def:asymb}. 
	
	\begin{definition}
		Let $i=\maxzero(\kUnd)$. Then, define $$D^\omega(\kUnd):=\bigsqcup_{j=i+1}^n\Slide^\omega(\kUnd^j).$$
	\end{definition}
	
	Then, we can piece together our maps $\sigmaT_{i,j}$ and $\sigmaT_j$ into one large map from $D^\omega(\kUnd)$ to $\Slwk$:
	
	\begin{definition}
		Let $i=\maxzero(\kUnd)$. Define $\Sigma_{\kUnd}:D^\omega(\kUnd)\to\Slwk$ by $$\Sigma_{\kUnd}(T):=\begin{cases}\sigmaT_{i,j}(T) & \text{if }T\in\Slide^\omega(\kUnd^j)\text{ for }k_j>1\\ \sigmaT_j(T) & \text{if }T\in\Slide^\omega(\kUnd^j)\text{ for }k_j=1 \end{cases}.$$
		\label{def:BigSigma}
	\end{definition}
	
	We now prove Theorem \ref{thm:bijectionExists}, which we restate below.
	
	\begin{bijectionExists}
		The map $\Sigma_{\kUnd}$ is a bijection.
	\end{bijectionExists}
	
	\begin{proof}
		By the definition of the asymmetric multinomial coefficients (see \cite{CGM} or Definition \ref{def:asymb}), $|D^\omega(\kUnd)|=|\Slwk|$. So, $\Sigma_{\kUnd}$ is a bijection if and only if it is injective.
		
		Let $T\in\Slwk$. Define $j=\last(T)$.
		By Lemma \ref{lem:larger_than_zero}, $j>\maxzero(\kUnd)$, so $j$ is a leaf that slides in $T$. By Theorem \ref{thm:lastadded}, if $T$ is in the image of either $\sigmaT_j$ or $\sigmaT_{i,j}$ for some $i$, then $\mbf j$ is the edge added by the map in question. In the latter case, $i$ must be $\maxzero(\kUnd)$.
		In particular, $T$ cannot be in the image of $\sigmaT_{j'}$ or $\sigmaT_{i,j'}$ for any $j'\neq j$.
		
		Note that if $T$ is in the image of $\sigmaT_j$ then $j$ slides only once, and if it is in the image of some $\sigmaT_{i,j}$ then $j$ slides more than once. So we can determine which case $T$ lies in by counting the number of edges labeled $\mbf j$.
		By Corollaries \ref{cor:ijInj} and \ref{cor:jInj}, both $\sigmaT_i$ and $\sigmaT_{i,j}$ are injective, so $T$ has at most one preimage under that particular $\sigmaT_j$ or $\sigmaT_{i,j}$, and none under any of the others. Thus, $\Sigma_{\kUnd}$ also is an injection, and so is a bijection.
	\end{proof}

	Finally, we define the inverse map $\Pi_{\kUnd}$ to $\Sigma_{\kUnd}$ where $\kUnd$ is a composition of $n$.
	
	\begin{definition}
		Let $i=\maxzero(\kUnd)$. Define $\Pi_{\kUnd}:\Slwk\to D^\omega(\kUnd)$ by $$\Pi_{\kUnd}(T):=\begin{cases}\piT_{i,\last(T)}(T) & \text{if }k_{\last(T)} >1\\ \piT_{\last(T)}(T) & \text{if }k_{\last(T)}=1 \end{cases}.$$
	\end{definition}
	
	It follows from the definitions of $\sigmaT_j$, $\piT_j$, $\sigmaT_{i,j}$, and $\piT_{i,j}$ that $\Sigma_{\kUnd}$ and $\Pi_{\kUnd}$ are inverses.
	
	\subsection{An explicit example of the bijection}
	
	So far, we have shown that $\Slwk$ satisfies the the asymmetric multinomial recursion \eqref{eq:recursion}, meaning we can define a bijection between $\Tour(\kUnd)$ and $\Slwk$ by unwinding both recurrences in parallel. However, it would be more satisfying to have a way of going between $\Tour(\kUnd)$ and $\Slwk$ directly. We will illustrate what this bijection looks like through a running example we will use throughout this subsection. This bijection will run through column-restricted parking functions and words as intermediate steps. We now describe how to read off a word from a slide tree.
	
	\begin{definition}
		First, for a word $w$, define $\bar{\sigma}_{i,j}(w)$ as the word obtained from $w$ by adding one to every entry $i$ or greater, then appending $j$ to the end, and $\bar{\sigma}_{j}(w)$ as the word obtained from $w$ by adding one to every entry $j$ or greater, then appending $j$ to the end. Then, we can define a map $$\word:\Slwk\to\text{words of content }\kUnd$$ as follows.
		$$\word(T)=\begin{cases} \emptyset &\text{ if }T=T_0\\ \bar{\sigma}_{i,j}(\word(T')) &\text{ if }T=\sigmaT_{i,j}(T')\\ \bar\sigma_{j}(\word(T')) &\text{ if }T=\sigmaT_{j}(T')\\
		\end{cases},$$ where $T_0$ denotes the unique slide tree in $\Slwk$ for $\kUnd$ the empty composition.
		\label{def:word()}
	\end{definition}
	
	In other words, the word for a tree $T$ is the slide labels of the internal edges, read off in the order the edges were added to the tree under the $\Sigma_{\kUnd}$ recursion.
	
	Meanwhile, for parking functions, a natural way to read off a word $w$ from a parking function $P$ is to let the value of $w_i$ be the column of $P$ that the number $i$ occurs in. As it turns out, there is a nice relationship between the words that can be obtained from slide trees and the words obtained from column-restricted parking functions.
	
	\begin{proposition}
		Let $\rev(w)=w_nw_{n-1}\cdots w_2w_1$ denote the reverse of a word $w$. Then, $w$ comes from an $\omega$ slide tree if and only if $\rev(w)$ comes from a column-restricted parking function.
		\label{prop:CPFrevSlide}
	\end{proposition}
	
	\begin{proof}
		If a word $w$ comes from a tree $T\in\Slwk$, the last letter $j$ of $w$ must be greater than $\maxzero(\kUnd)$, since $T$ is in the image of some $\sigmaT_{i,j}$ or $\sigmaT_j$. Similarly, the dominance condition on column-restricted parking functions forces the 1 to be left of all empty columns, which in turn means the first entry in the corresponding CPF-word must be greater than $\maxzero(\kUnd)$. From here, it is clear from the recursions $\Sigma_{\kUnd}$ above in Definition \ref{def:BigSigma} and the map $\varphi$ in Section 5 of \cite{CGM} result in words that are reverses of each other.
	\end{proof}
	
	We now illustrate an example of the full sequence of maps from $\Tour(\kUnd)$ to $\Slwk$.
	
	\begin{example} 
		
		Let $T\in\Tour(0,0,1,1,2,0,3,1)$ be the tournament tree below. 
		
		\begin{center}
			\includegraphics[scale=0.8]{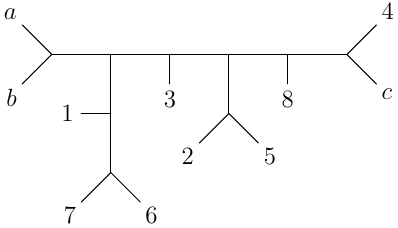}
		\end{center}
		
		Using the map $\tau:\Tour(\kUnd)\to\CPF(\kUnd)$ from \cite{GGL23}, this corresponds to the following column-restricted parking function.
		\begin{center}
			\begin{tikzpicture}[scale=0.6]
			\draw (0,0)--(8,-8);
			\draw[thick] (0,0)--(3,0)--(3,-1)--(4,-1)--(4,-2)--(5,-2)--(5, -4)--(7,-4)--(7,-7)--(8,-7)--(8,-8);
			\node() at (2.75,-0.5) {$\mbf7$};
			\node() at (3.75,-1.5) {$\mbf4$};
			\node() at (4.75,-2.5) {$\mbf6$};
			\node() at (4.75,-3.5) {$\mbf2$};
			\node() at (6.75,-4.5) {$\mbf8$};
			\node() at (6.75,-5.5) {$\mbf3$};
			\node() at (6.75,-6.5) {$\mbf1$};
			\node() at (7.75,-7.5) {$\mbf5$};
			\end{tikzpicture}
		\end{center}
		
		Next, reading off from this parking function, we get the word $75748537$. Taking the reverse of this results in the word $73584757$, which is what we will use to construct our slide tree.  To do this, we have
		\begin{align*}
		\word^{-1}(73584757) &= \sigmaT_{6,7}(\word^{-1}(6357465))= \sigmaT_{6,7}(\sigmaT_{2,5}(\word^{-1}(524635)))\\ &= \sigmaT_{6,7}(\sigmaT_{2,5}(\sigmaT_{1,5}(\word^{-1}(41352)))) = \sigmaT_{6,7}(\sigmaT_{2,5}(\sigmaT_{1,5}(\sigmaT_2(\word^{-1}(3124 )))))\\ &= \sigmaT_{6,7}(\sigmaT_{2,5}(\sigmaT_{1,5}(\sigmaT_2( \sigmaT_4(\word^{-1}(312))))))\\ &= \sigmaT_{6,7}(\sigmaT_{2,5}(\sigmaT_{1,5}(\sigmaT_2( \sigmaT_4(\sigmaT_2(\word^{-1}(21)))))))\\&= \sigmaT_{6,7}(\sigmaT_{2,5}(\sigmaT_{1,5}(\sigmaT_2( \sigmaT_4(\sigmaT_2(\sigmaT_1(\word^{-1}(1))))))))\\&= \sigmaT_{6,7}(\sigmaT_{2,5}(\sigmaT_{1,5}(\sigmaT_2( \sigmaT_4(\sigmaT_2(\sigmaT_1(\sigmaT_1(T_0)))))))).
		.\end{align*}
		
		So, we start with adding the edge $\mbf1$ to the empty tree. Then, the next two steps, $\sigmaT_1$ and $\sigmaT_2$, do not require us to move where we add an edge, so we add the edges $\mbf1$ and $\mbf2$ to the right of the previous steps, respectively. These first three steps are shown below.
		
		\begin{center}
			\includegraphics[scale=0.8]{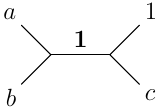}
			\hspace{1cm}
			\includegraphics[scale=0.8]{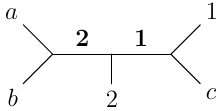}
			\hspace{1cm}
			\includegraphics[scale=0.8]{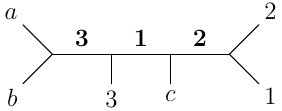}
		\end{center}
		
		Similarly, since $4>2$, $\sigmaT_4$ can again add 4 to the end with no complications.
		
		\begin{center}
			\includegraphics[scale=0.8]{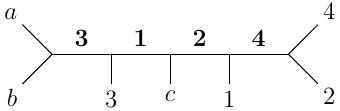}
		\end{center}
		
		However since $2<3<5$, $\sigmaT_{2}$ is in the third case of Step 2 of Definition \ref{def:sigmaJ}. So, we consider the decreasing sequences $(3,1)$, $(2)$, and $(4)$ from $a$ to $4$, and conclude that we add the edge $\mbf2$ off of the leaf 1.
		
		\begin{center}
			\includegraphics[scale=0.8]{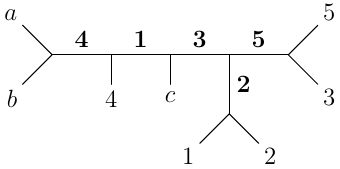}
		\end{center}
		
		Next, $\sigmaT_{1,5}$ adds an edge $\mbf 5$ and leaf 1. Since the leaf 4 only has one internal edge between it and $a$, there is only one branch $B_1$ to consider, which has minimal leaf $c$. Thus, we are in Case 1 of Step 5 of Definition \ref{def:sigmaIJ}. Thus, we add the new edge and leaf right where the old leaf 5 (relabeled from 4) was.
		
		\begin{center}
			\includegraphics[scale=0.8]{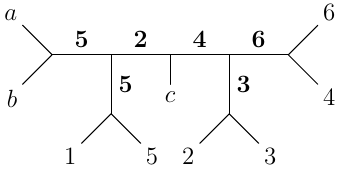}
		\end{center}
		
		For $\sigmaT_{2,5}$, we have decreasing sequences $(6,3)$, $(5)$, and $(7)$ (after relabeling). So, since $c<2<3<7$, we are in Case 3, so replace the leaf 3 with the leaf 2, and move the three to the end of the new edge $\mbf5$. Note that the relabeling of the leaves makes it look look like the leaf 2 does not move, but the new 3 is actually the relabeled 2.
		
		\begin{center}
			\includegraphics[scale=0.8]{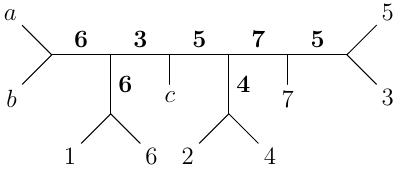}
		\end{center}
		
		Finally, we apply $\sigmaT_{6,7}$. Since $6$ is larger than $c$ and $1$, we are again in Case 1. So, we add the leaf 6 to the end of the new edge $\mbf7$.
		
		\begin{center}
			\includegraphics[scale=0.8]{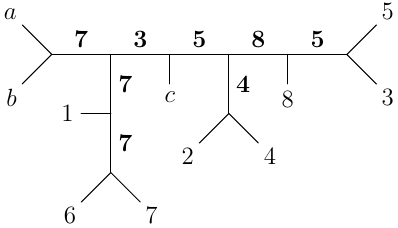}
		\end{center}
		
		Thus, the above tree $T'\in\Slide^\omega(0,0,1,1,2,0,3,1)$ is the slide tree in bijection with the tournament tree $T$ we started with.
	\end{example}

	\section{Caterpillar trees and pattern avoidance}
	\label{sec:cats}
	
	In this section, we discuss specific case of  caterpillar trees.  Recall the definition of caterpillar trees from Definition \ref{def:Caterpillar}.
	
	As restated in Proposition \ref{prop:catAvoid23-1}, it was found in \cite{GGL22} that in the case where $\kUnd=(1,1,\ldots,1)$, the words obtained from caterpillar trees are precisely the set of $23{-}1$ avoiding permutations. In other words, they found a bijection $\Cat^\psi(\kUnd)\leftrightarrow\Av_n(23{-}1)$. We wish to describe a similar correspondence in terms of pattern avoidance for any composition $\kUnd$. 
	
	If we consider what the map $\word()$ in Definition \ref{def:word()} does to caterpillars, we see that it will always read off the slide labels from left to right, in order. This coincides with the notion of reading edge labels off of a tree from the root to the right that we discussed in Section \ref{sec:preliminaries}.
	
	We first show that every word comes from at most one caterpillar slide-tree, and define a map from words to trivalent, leaf-labeled trees that outputs this unique tree for a particular tree, if such a tree exists. Then, we give such a correspondence for the case where $\kUnd$ is right-justified, as that is precisely when the caterpillar words can be described solely by a pattern avoidance condition. After that, we give the more general characterizations for all cateprillar slide trees, first for $\Cat^\psi(\kUnd)$, and then for $\Cat^\omega(\kUnd)$.
	
	\subsection{Preliminaries on caterpillar slide trees}
	
	Before we describe how to go between caterpillars and words, we first show that such a correspondence is well-defined. That is, we first show that each word $w$ can come from at most one caterpillar tree, and then describe which words do arise as edge words of caterpillar trees. We define an algorithm that takes a word $w$ and constructs the only caterpillar tree $\tree(w)$ that could have $w$ as its word of edge labels. That is, $\tree(w)$ is always a trivalent caterpillar tree with labeled leaves, and if $\tree(w)$ is a slide tree, then $w$ is its edge word.
	
	\begin{definition}
		Given a word $w$ of content $\kUnd$ of length $n$, we construct the caterpillar tree $\tree(w)$ as follows.
		\begin{enumerate}
			\item Construct a path graph of length $n$, and add leaf edges to it to make it trivalent.
			\item By convention, we treat the left end as the root and label the two leaves on the left end by $a$ and $b$.
			\item Label the internal edges from left to right using the letters of $w$ in order.
			\item For each pair of adjacent internal edges, let $\mbf x$ be the label of the left edge and $\mbf y$ the label of the right edge (that is, $\mbf x$ is closer to the root). We call the leaf between them a \textbf{descent} leaf if $\mbf x>\mbf y$, and a \textbf{nondescent} leaf if $\mbf x\leq \mbf y$. For the two leaves on the right end, we define one to be a descent leaf and the other a nondescent leaf.
			\item Label each descent leaf with the label of the edge immediately to its left.
			\item Create a list of the remaining unused labels among $c,1, \ldots,n$.
			\item From left to right, label each nondescent leaf by the smallest remaining unused label, unless that label labels the edge immediately to the right of that leaf. In that case, instead label the leaf with the second smallest unused label.
			\item Forget the edge labels.
		\end{enumerate}
		This forms a leaf-labeled trivalent caterpillar tree $\tree(w)$.
		\label{def:tr}
	\end{definition}
	
	\begin{example}
		Consider when $w=666224$. After Step 3, our tree looks like this:
		\begin{center}
			\includegraphics[scale=0.8]{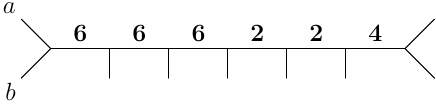}
		\end{center}
		Then, we label only the descent vertices. The only descent is after the last 6, and there is also a descent leaf on the end, so we label the 6 and the 4:
		\begin{center}
			\includegraphics[scale=0.8]{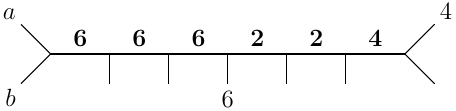}
		\end{center}
		Then, we work from left to right, labeling the nondescent leaves:
		\begin{center}
			\includegraphics[scale=0.8]{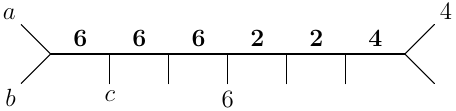}
			\hspace{0.25cm}
			\includegraphics[scale=0.8]{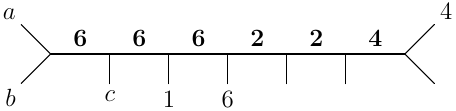}
		\end{center}
		Up until now, we have been able to use the smallest unused label at each step. Now however, the smallest unused label is $2$, but since $2$ also labels the edge to the right of the next leaf, we skip 2 and instead use 3, the second smallest remaining label:
		\begin{center}
			\includegraphics[scale=0.8]{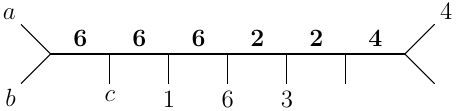}
		\end{center}
		Now we are free to use the 2 for the following leaf, and can finish off the leaves:
		\begin{center}
			\includegraphics[scale=0.8]{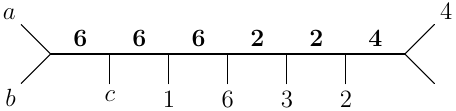}
			\hspace{0.25cm}
			\includegraphics[scale=0.8]{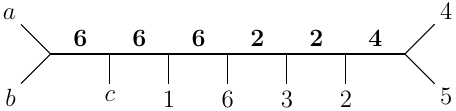}
		\end{center}
		Finally, we remove the edge labels to form $\tree(w)$:
		\begin{center}
			\includegraphics[scale=0.8]{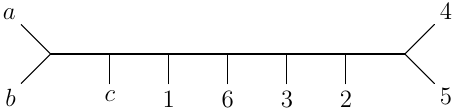}
		\end{center}
		Note that in this case, $\tree(666224)\in\Slide^\psi(0,2,0,1,0,3)\setminus\Slide^\omega(0,2,0,1,0,3)$.
		
	\end{example}
	
	We wish to characterize when $\tree(w)$ is a slide tree. One necessary condition for being a slide tree is that $\tree(w)$ must use each leaf label $a,b,c,1,2,\ldots,n$ exactly once. This happens precisely when $w$ avoids $2{-}1{-}2$, which we show in  the following lemma.
	
	\begin{lemma}
		When $w\in\Av_{\kUnd}(2{-}1{-}2)$, Definition \ref{def:tr} produces a tree that uses each leaf label $a,b,c,1,2,\ldots,n$ exactly once.
	\end{lemma}
	
	\begin{proof}
		If $w$ contains no $2{-}1{-}2$ pattern, then no label $\mbf i$ can be used to label two different descent leaves in Step 5, since all edges between two edges with label $\mbf i$ must be weakly larger than $\mbf i$, meaning only the rightmost $\mbf i$ can possibly be left of a descent leaf. So, Step 5 uses each label at most once. Then, Step 7 clearly never uses any label for a second time. There are $n+3$ leaves and $n+3$ labels total, and the last nondescent leaf on the end of the tree can always use the smallest (and only) remaining label, as there is no edge to its right. So, we use every leaf label $a,b,c,1, \ldots,n$ exactly once.
	\end{proof}
	
	In other words, $\tree(w)$ is well-defined as a map from $\Av_{\kUnd}(2{-}1{-}2)$ to the set of leaf-labeled trivalent caterpillar trees using the labels $a,b,c,1, 2,\ldots,n$.
	
	Given a particular caterpillar slide tree, one can also read off the word formed along its edges by the slide-labeling algorithm.
	
	\begin{theorem}[Uniqueness of slide trees]
		Each word comes from at most one caterpillar $\psi$-slide tree.
	\end{theorem}
	
	\begin{proof}
		We show that given any word $w$, if there is a corresponding slide tree, it must be $\tree(w)$.
		
		First, we observe that if a tree $\tree(w)$ is a slide tree, the descent leaves are the positions where the edge to the left is contracted before the edge to the right in the $\kUnd$-slide labeling algorithm, while the nondescent leaves are the positions where the right edge is contracted before the left edge.
		
		Given a word $w$, any corresponding slide tree must have its descent leaves labeled as in Definition \ref{def:tr}. Otherwise, if $\mbf x$ is the edge label on the left of a descent leaf not labeled this way, then the leaf label $x$ will not label that edge $\mbf x$.
		
		Finally, for each nondescent leaf $v$, when the edge on its right is labeled by some label $\mbf i$, the $\kUnd$-slide algorithm compares the label of $v$ against the set of all leaf labels to the right of $v$ except for $i$. So, the nondescent labels must be labeled as in Step 7, so that it is possible that each nondescent leaf is smaller than everything to its right, expect possibly for the label doing the slide.
	\end{proof}
	
	\begin{corollary}
		Due to Lemma \ref{lem:omega<psi}, this also gives us uniqueness for $\omega$-slide trees.   
	\end{corollary}
	
	\begin{lemma}
		Let $v$ be a nondescent leaf of a caterpillar slide tree $T\in\Slide^\psi(\kUnd)$, and $\mbf e$ the edge on its right. If $v>\mbf e$, there is no edge labeled $v$.
		\label{lem:EverSlide}
	\end{lemma}
	
	\begin{proof}
		Let $\mbf d$ be the edge to the left of $v$, and assume that $v>\mbf e$. Then, since $v$ is a nondescent leaf, $\mbf d\leq\mbf e$, so $\mbf d<v$. Suppose that $v$ labels some edge during the slide algorithm. Since $v>\mbf d$, $v$ would label edges before $\mbf d$, and so the edge on the left of $v$ would be labeled by something at least as large as $v$, and thus would already be contracted when $\mbf d$ did its labeling, meaning $\mbf d$ could not label it. This contradicts our assumption that $\mbf d$ labels the edge on the left, so $v$ must label no edges.
	\end{proof}
	
	Recall Definitions \ref{def:barred} and \ref{def:vincular}. For a word $w$ to avoid the pattern $23{-}\overline{2}{-}1$, then whenever there are indices $i,j$ such that $i+1<j$ and $w_j<w_i<w_{i+1}$, there must be some $i+1<k<j$ such that $w_k=w_i$. Intuitively, this means that every $23{-}1$ pattern must extend to a $23{-}2{-}1$ pattern.
	
	\begin{example}
		The word $w=35432$ avoids the pattern $23{-}\overline{2}{-}1$. Even though the subword $352$ forms a $23{-}1$ pattern, it extends to the subword $3532$, so this is not an instance of a $23{-}\overline{2}{-}1$ pattern.
	\end{example}
	
	\begin{lemma}
		Let $T$ be a tree in $\Cat^\psi(\kUnd)$. Then, the word $w$ formed by reading off the edges starting from the root avoids the patterns $2{-}1{-}2$ and $23{-}\overline{2}{-}1$.
		\label{lem:avoids}
	\end{lemma}
	
	\begin{proof}
		Clearly, $w$ can not contain a $2{-}1{-}2$ pattern, since in each iteration in the $\kUnd$-slide algorithm, after contracting the already labeled edges, all labels of the same value must be consecutive.
		
		For $23{-}\overline{2}{-}1$, suppose instead that $w_i=x$, $w_{i+1}=y$, and $w_j=z$ form a $23{-}\overline{2}{-}1$ pattern. Then, $w_i$ is the rightmost instance of $x$ in $w$, since this assumes there is no $x$ between $w_i$ and $w_j$, and we already showed that $w$ avoids $2{-}1{-}2$, so there is no $x$ to the right of $w_j$. Next we have $x<y$, so the leaf labeled $x$ is not a descent leaf, and is instead a nondescent leaf. Since $x$ does slide at some point in $T$, the leaf $x$ must be smaller than the edge to its right by Lemma \ref{lem:EverSlide}. Then, since that edge on the right of $x$ is a valid slide, $x$ must be smaller than every leaf to its right. However, the leaf $x$ must be to the left of $\mbf{w_j}$ in order for it to label $\mbf{w_i}$, and the leaf $z$ must be to the right of $\mbf{w_j}$. This is a contradiction, so $w$ avoids $23{-}\overline{2}{-}1$.
	\end{proof}
	
	\begin{lemma}
		Let $\kUnd$ be a reverse-Catalan composition of $n$ with $m$ zeros, such that $$w\in\Av_{\kUnd}(2{-}1{-}2,\ 23{-}\overline{2}{-}1).$$ Let $i_1<\ldots<i_l$ be the indices of the entries in $w$ that are strictly smaller than all entries to their right. Then, in $\tree(w)$, each edge $\mbf{w_{i_j}}$ has a nondescent leaf immediately to its right, and there are exactly $m$ additional nondescent leaves.
		\label{lem:numb2}
	\end{lemma}
	
	\begin{example}
		Let $w=73152587$. Then, $l=4$, $i_1=3$, $i_2=5$, $i_3=6$, $i_4=8$, and $w_{i_1}=1$, $w_{i_2}=2$, $w_{i_3}=5$, and $w_{i_4}=7$.
		\begin{center}
			\begin{tikzpicture}[scale=0.3]
			\draw (0.5,0.5)--(9.5,0.5)--(9.5,9.5)--(0.5,9.5)--(0.5,0.5);
			\fill[black] (1.5,7.5) circle (0.25cm);
			\fill[black] (2.5,3.5) circle (0.25cm);
			\fill[black] (3.5,1.5) circle (0.25cm);
			\fill[black] (4.5,5.5) circle (0.25cm);
			\fill[black] (5.5,2.5) circle (0.25cm);
			\fill[black] (6.5,5.5) circle (0.25cm);
			\fill[black] (7.5,8.5) circle (0.25cm);
			\fill[black] (8.5,7.5) circle (0.25cm);
			\draw (3.5,1.5) circle (0.5cm);
			\draw (5.5,2.5) circle (0.5cm);
			\draw (6.5,5.5) circle (0.5cm);
			\draw (8.5,7.5) circle (0.5cm);
			\draw (9,7)--(7,7)--(7,5)--(6,5)--(6,2)--(4,2)--(4,1)--(1,1);
			\end{tikzpicture}
		\end{center}
	\end{example}
	
	\begin{proof}[Proof (of Lemma \ref{lem:numb2})]
		First, note that there are exactly $m$ repeats in $w$ (that is, a position where the letter also occurs again later in $w$). We show that the nondescent leaves of $\tree(w)$ correspond precisely to the position of these repeats along with the positions $w_{i_1},\ldots,w_{i_{l-1}}$.
		
		Consider a pair of consecutive entries $x=w_j$ and $y=w_{j+1}$ of $w$. If $x$ and $y$ form a descent, then the leaf between them in $\tree(w)$ is a descent leaf. Then, $j$ is not some $i_s$ and $x$ is not a repeat, since otherwise $x$ and $y$ would form the first two entries of a $2{-}1{-}2$ pattern.
		
		Otherwise, if $x$ and $y$ form a nondescent, then the leaf between them in $\tree(w)$ is a nondescent leaf. If $x=y$, then clearly $x$ is a repeat and is not some $w_{i_s}$. If $x$ and $y$ form an ascent, there are three cases: $j\in\{i_1,\ldots,i_l\}$, $x=w_{i_s}$ for some $i_s>j$, or $x>w_{i_s}$ for some $i_s>j$. In the second case, clearly $x$ is a repeat of $w_{i_s}$. In the third case, $x$ must also be a repeat, or else $x$, $y$, and $w_{i_s}$ form a $23{-}\overline{2}{-}1$ pattern. 
		
		Finally, there is a nondescent leaf to the right of $w_{i_l}$ by construction. Thus, we have shown that the nondescent leaves are on the right of each edge corresponding to some $w_{i_j}$ and each repeated entry of $w$.
	\end{proof}
	
	\begin{corollary}
		Let $\kUnd$, $w$, and $i_1,\ldots,i_l$ be as in Lemma \ref{lem:numb2}. Then for each $j$, $i_j$ has at most $m+j-1$ nondescent leaves to its left.
		\label{cor:MAXm+j-1}
	\end{corollary}
	
	\subsection{Right-justified compositions}
	\label{subsec:right-justified}
	
	Now, we consider the case of compositions for which the caterpillar trees are characterized solely by a pattern avoidance criterion: That of right-justified compositions.
	
	\begin{definition}
		A composition $\kUnd$ of $n$ is \textbf{right-justified} if all zeros occur before all non-zero entries.
	\end{definition}
	
	\begin{remark}
		When performing the $\psi$-slide algorithm on a caterpillar tree $\tree(w)$, in every comparison, the left element compared will be a nondescent leaf. This is because every descent leaf has its left edge labeled before its right and the left edge is only labeled if there is a smaller leaf to its left. Thus, a descent leaf can not be compared in this way.
	\end{remark}
	
	\begin{rjust}
		Let $\kUnd$ be right-justified with $m$ zeros, and let $w\in\Av_{\kUnd}(2{-}1{-}2,\ 23{-}\overline{2}{-}1)$. Then, $\tree(w)$ is a valid slide tree.
	\end{rjust}
	
	\begin{proof}
		Let $i_1,\ldots,i_l$ be as before. Since $\kUnd$ is right-justified, the smallest letter present in $w$ is $w_{i_1}=m+1$, so in particular, $c,1,\ldots,m$ are not used as descent leaf labels, and thus are used as nondescent labels. We first prove the claim that during Step 7 of Definition \ref{def:tr}, we will always use the smallest remaining label for each nondescent leaf. In particular, this results in the nondescent leaf labels being in increasing order, as read from left to right.
		
		We proceed by strong induction from left to right in the tree. For the base case, clearly we can always use $c$ first, since there is never an edge labeled $c$. For the induction step, assume this is true for all nondescent leaves to the left of a nondescent leaf $v$. 
		
		We first consider the case where $v$ is just to the right of an edge $\mbf{w_{i_j}}$. Let $\mbf y$ be the edge on the right of $v$. Note that $\mbf {w_{i_j}}<\mbf y$ by definition of $w_{i_j}$. We wish to show that $v<\mbf y$. In fact, we will show that in the step of Definition \ref{def:tr} where we label $v$, $w_{i_{j-1}}$ (or $m$ in the case where $j=1$) has not yet been used as a leaf label. By Corollary \ref{cor:MAXm+j-1}, $v$ has at most $m+j-1$ nondescent leaves strictly to its left. We also know that  $c,1,2,\ldots,m,w_{i_1}, \ldots,w_{i_{j-1}}$ are all used as nondescent labels. This means that there are at least $m+j$ nondescent labels smaller than $w_{i_j}$. By the inductive hypothesis, all nondescent labels thus far have been used in increasing order, so in particular, the largest such label, $w_{i_{j-1}}$ (or $m$ when $j=1$), has not yet been used. Since additionally $w_{i_{j-1}}\neq y$ ($m\neq y$), $w_{i_{j-1}}$ ($m$) can be used to label $v$. This completes the inductive step for the case where $v$ is just to the right of some $w_{i_j}$.
		
		Otherwise, $v$ is a nondescent leaf not on the right side of an edge labeled $w_{i_j}$. Let $w_{i_j}$ be the rightmost among the $w_{i_r}$'s to the left of $v$ (if such a $w_{i_j}$ exists). Again, the edge $\mbf y$ to the right of $v$ is larger than $w_{i_j}$ by the definition of $w_{i_j}$. Notice that since $w_{i_{j+1}}$ (which exists because $i_l=n$) cannot have more than $m+j$ nondescent leaves to its left, $v$ has no more than $m+j-1$ nondescent leaves to its left. The argument then follows identically to the previous case. This completes the inductive step. Thus, we have shown that during Step 7 of Definition \ref{def:tr}, we will always use the smallest remaining label for each nondescent leaf.
		
		Next, we show that for any leaf in the tree, the smallest leaf to its right is a nondescent leaf. To do this, it suffices to show that each descent leaf has a smaller nondescent leaf to its right. So, let $d$ be some descent leaf. The edge on its right is smaller than $d$, so $d$ is not any of the $w_{i_j}$. Let $j$ be the smallest index such that $\mbf{w_{i_j}}$ is to the right of $d$, and $e$ the leaf on the right of $\mbf{w_{i_j}}$. If $d<w_{i_j}$, then since $d$ is not a $w_{i_r}$, there must be another one between $d$ and $w_{i_j}$, which contradicts the assumption that $w_{i_j}$ is the closest one to the left. So, $d>\mbf{w_{i_j}}\geq e$, where $e$ is a nondescent leaf label by Lemma \ref{lem:numb2}. Thus, for any internal edge, the smallest leaf label to its right is a nondescent leaf.
		
		Thus far, we have shown that for any internal edge, the smallest leaf label to its right is a nondescent leaf, and the nondescent leaves are in increasing order from left to right. So, during the slide algorithm, each comparison compares a nondescent leaf somewhere to the right against a nondescent leaf to the left, which must be smaller it. So, assuming that the leaf label is in the correct position to slide, the comparison made for that edge is valid, and so the slide is successful. 
		
		To show that leaf labels are always in the correct positions to slide, it suffices to show that every nondescent leaf label is to the right of the rightmost edge with that label (if any), and that any edges between them are larger. The set of nondescent labels that are also edge labels are precisely $w_{i_1},\ldots,w_{i_{l-1}}$. So, by Corollary \ref{cor:MAXm+j-1}, each such label labels a leaf strictly to the right of the rightmost edge with that label, and all the intermediate edges are larger by the definition of the $w_{i_j}$'s. Thus, every slide is a valid $\psi$-slide, and $\tree(w)\in\Slide^\psi(\kUnd)$.
	\end{proof}
	
	Then, as an immediate consequence of Theorem \ref{thm:rjust} and Lemma \ref{lem:avoids}, we get our desired result.
	
	\begin{corollary}
		Let $\kUnd$ be a right-justified composition. Then, the map $$\tree: \Av_{\kUnd}(2{-}1{-}2,\ 23{-}\overline{2}{-}1)\to\Cat^\psi(\kUnd)$$ is a bijection.
	\end{corollary}
	
	\subsection{Pattern avoidance conditions for $\Cat^\psi(\kUnd)$}
	
	In this subsection, we characterize all the caterpillar trees in $\Slpk$, including when $\kUnd$ is not right-justified.

	\begin{definition}
		For a word $w$ with content $\kUnd$, let $k(i,j)$ denote the number of $j$'s in the subword of $w$ right of the rightmost $i$. Then, define the total number of repeats right of the rightmost $i$ by $$\TRep_w(i)=\sum_{j\neq i}\max((k(i,j)-1),0).$$
	\end{definition}
	
	\begin{definition}
		For a word $w$, let $\ell_w(i)$ denote the total number of consecutive $i$'s in the right-most consecutive sequence of $i$'s in $w$. Let $\mathbf{i_l}$ and $\mathbf{i_r}$ denote the leftmost and rightmost such entries, respectively.
	\end{definition}
	
	\begin{example}
		Let $w=313321$. Then, $\ell_w(1)=1$, $\ell_w(2)=1$, and $\ell_w(3)=2$. Additionally, $\mathbf{3_l}$ is in position 3 and $\mathbf{3_r}$ is in position 4.
	\end{example}

	\begin{lemma}
		Let $\kUnd$ be a reverse-Catalan composition, and let $w$ be a word of content $\kUnd$. If $\TRep_w(i)+\ell_w(i)<\z(i)$ for some $i$, then there is some $j$ for which $\TRep_w(j)+\ell_w(j)<\z(j)$ and the subword of $w$ right of $\mathbf{j_r}$ contains no entries smaller than $j$.
		\label{lem:rightmost_i_Psi}
	\end{lemma}
	
	\begin{proof}
		For such an $i$, let $j$ be the rightmost entry in $w$ (weakly) smaller than $i$. If $i=j$, we are done. Otherwise, since $j<i$, $\z(j)\geq \z(i)$. In addition, since $j$ is right of every $i$ by definition, every repeat right of all $j$'s is also right of all $i$'s. All but one of the $\ell_w(j)$ copies of $j$ is counted by the term $\TRep(i)$, so $\TRep_w(j)+\ell_w(j) \leq \TRep_w(i)+\ell_w(i)$. Thus, we have $$\TRep_w(j)+\ell_w(j)\leq \TRep_w(i)+\ell_w(i)<\z(i)\leq \z(j).$$ So, $j$ is an index for which $\TRep_w(j)+\ell_w(j)<\z(j)$, and by construction the subword of $w$ right of $\mathbf{j_r}$ contains no entries smaller than $j$.
	\end{proof}
	
	\begin{lemma}
		Let $T=\tree(w)$, and let $i$ be a nondescent leaf for which $k_i\neq0$. Then, the edge next to $\mbf{i_r}$ on the right is larger than $i$.
		\label{lem:descEdgeDesc}
	\end{lemma}
	
	\begin{proof}
		Since $i$ is a nondescent leaf, that means that it did not label a leaf in Step 5 of Definition \ref{def:tr}. More specifically, the leaf $v$ on the right of $\mbf{i_r}$ was not labeled at this step, since only $i$ could have labeled it at this stage. So, $v$ is a nondescent leaf too, leaning the edge on its right is larger than $i$.
	\end{proof}

	\begin{definition}
		Consider a tree $T=\tree(w)$. Fix an $i$ for which $k_i\neq0$.
		
		Let $\rnd$ (Right Non-Descents) denote the number of nondescent leaves strictly to the right of $\mathbf{i_r}$, and $\lnd$ (Left Non-Descents) the number of nondescent leaves to the left of $\mathbf{i_l}$. 
		
		Let $z_-$ denote the number of $j$ for which $k_j=0$ and $j<i$, and $z_+$ the number of $j$ for which $k_j=0$ and $j>i$. 
		
		Then, let $p_+$ be the number of nondescent leaves $j$ which do occur as edge labels in $T$ and $j\geq i$. Similarly, let $p_-$ be the number of nondescent leaves $j$ which do occur as edge labels in $T$ and $j<i$. 
	\end{definition}
	
	Note that in the above definition, $z_+=z(i)$ from Definition \ref{def:z(i)}.
	
	Using these quantities, we will count the number of nondescent leaves of a tree, relative to a particular choice of $i$, in two different ways.
	
	\begin{lemma}
		Let $\kUnd$ be a reverse-Catalan composition, $w$ be a word of composition $\kUnd$, and $T=\tree(w)$. Fix an $i\in[n]$ such that $k_i\neq0$. Then,
		\begin{equation}
		\lnd+\rnd+\ell_w(i)-1=z_++z_-+p_+ +p_-+1.
		\label{eq:leaves}
		\end{equation}
		
		\label{lem:2leaves}
	\end{lemma}
	
	\begin{proof}
		We first will count the nondescent leaves of $T$ in terms of their position in the tree. 
		
		First, $\ell_w(i)$ as defined above is the number of edges in the rightmost consecutive group of $\mathbf{i}$'s. So, there are $\ell_w(i)-1$ nondescent leaves between $\mathbf{i_l}$ and $\mathbf{i_r}$. Then, $\lnd$ and $\rnd$ denote the number of nondescent leaves left of $\mathbf{i_l}$ and right of $\mathbf{i_r}$, respectively, by definition. So, the total number of nondescent leaves is $\lnd+\rnd+\ell_w(i)-1$.
		
		Alternatively, we can count the nondescent leaves in terms of their label values. First, any $j$ for which $k_j=0$ is a nondescent leaf, which contributes $z_++z_-$. Next, we have $p_++p_-$ nondescent leaves $j$ which do occur as edge labels, split up by whether and $j\geq i$ or $j<i$. 
		Finally, $c$ is also a nondescent leaf not included in the above categories. Since every leaf label either does label an edge or it does not, we have that the total number of nondescent leaves is $z_++z_-+p_+ +p_-+1$.
		
		Thus, we have proven \eqref{eq:leaves}.
	\end{proof}
	
	\begin{lemma}
		The $p_+$ nondescent leaves that slide in $T$ and are weakly larger than $i$ all occur to the right of $\mbf{i_r}$.
		\label{lem:largerToRight}
	\end{lemma}
	
	\begin{proof}
		For $j=i$, clearly the leaf $i$ appears right of $\mbf{i_r}$ (regardless of whether $i$ is a nondescent leaf). Then, note that for all such $j>i$, the leaf $j$ must be to the right of $\mbf{i_r}$, or else there is an edge $\mbf j$ left of $\mbf{i_r}$, and the rightmost such edge $\mathbf{j}$ and the edge on its right (which is larger than $j$ by Lemma \ref{lem:descEdgeDesc}) would form a $23{-}\overline{2}{-}1$ pattern with $\mathbf{i_r}$.
	\end{proof}
	
	Now, we prove the $\psi$ version of Theorem \ref{thm:psiAvoidance}.
	
	\begin{theorem}
		Let $\kUnd$ be a reverse-Catalan composition of $n$, and let $w$ be a word of composition $\kUnd$. Then, $\tree(w)\in\Cat^\psi(\kUnd)$ if and only if $w\in\Av_{\kUnd}(2{-}1{-}2, 23{-}\overline{2}{-}1)$ and 
		\begin{equation} 
		\TRep_w(i)+\ell_w(i)\geq\z(i)
		\label{eq:psiCat}
		\end{equation}
		for all $i$ such that $k_i>0$.
	\end{theorem}
	
	\begin{proof}
		($\implies$) The pattern avoidance condition is given by Lemma \ref{lem:avoids}.
		
		We will prove the second condition by the contrapositive. That is, suppose $\tree(w)\in\Cat^\psi(\kUnd)$, $w\in\Av_{\kUnd}(2{-}1{-}2, 23{-}\overline{2}{-}1)$, but there is some edge label $i$ for which $w$ does not satisfy \eqref{eq:psiCat}. By Lemma \ref{lem:rightmost_i_Psi}, we may assume that no edges smaller than $i$ occur to the right of $\mathbf{i_r}$.
		
		By definition, we have $z(i)=z_+$. Next, consider the branch $B$ starting at $\mathbf{i_r}$, and recall that any branch of a tree has one more leaf than internal edge.  Let $|B|$ denote the number of internal edges in this branch.  Then if $D$ is the number of descent leaves in $B$, we have $$\rnd+D=|B|+1.$$
		
		Each descent leaf $x$ is just to the right of the rightmost $\bf{x}$ by definition.  So, rewriting the above equation as $\rnd-1=|B|-D$ we find that there are $\rnd-1$ remaining edges, which all are either repeats of edge labels to their right, or are the rightmost edges with a label, where the label is a nondescent leaf to the right of $i$.  We also assumed that no edges smaller than $\mathbf{i}$ occur right of $\mathbf{i_r}$, so $\rnd-1$ equals the total number of repeats right of $\mathbf{i_r}$ plus the number of nondescent leaves in $\tree(w)$ larger than $i$ that label an edge (which by Lemma \ref{lem:largerToRight} all occur right of $\mathbf{i_r}$). Symbolically, $\rnd-1=\TRep_w(i)+p_+$, or equivalently, $\TRep_w(i)=\rnd-p_+-1$. Using our assumption that \eqref{eq:psiCat} is not satisfied, and applying Lemma \ref{lem:2leaves}, we have: \begin{align*} \TRep_w(i)+\ell_w(i) &< z(i)\\ \rnd-p_+-1+\ell_w(i) &< z_+\\ \rnd+\ell_w(i)-z_+-p_+ &< 1\\ z_-+p_--\lnd+2 &< 1\\ z_-+p_-+1 &< \lnd. \end{align*} 
		
		Thus, the total number of nondescent leaves left of $\mathbf{i_l}$ is greater than the number of nondescent leaf labels smaller than $i$ (including $c$). Since the caterpillar labeling algorithm labels nondescent leaves from left to right, and always chooses one of the two smallest unused labels, there is a nondescent leaf $d\geq i$ to the left of $\mathbf{i_l}$. If $d=i$, then $i$ cannot label the edges $\mathbf{i_l}$ through $\mathbf{i_r}$, which contradicts our assumption. So, we may assume $d>i$. Let $\mathbf{j}$ be the edge immediately to the left of $\mathbf{i_l}$. By definition of $\mbf{i_l}$, $\mathbf{j}\neq i$. We now consider the two cases of whether or not $\mathbf{j}$ is larger than $i$:
		
		\textbf{Case 1:} $\mathbf{j}>i$. In this case, $\mathbf{j}$ must be a descent leaf, which in particular means that $d$ is to the left of $\mathbf{j}$ as well. Let $\mathbf{e}$ be the edge to the right of $d$. If $\mathbf{e}\neq i$, then in the slide algorithm $\mathbf{e}$ will compare $i$ (or something smaller if a smaller leaf than $i$ exists right of $\mathbf{e}$) to $d$, and since $d>i$ this would be an invalid slide. So, we must have $\mathbf{e}=i$. Then, since $\mathbf{e}<\mathbf{j}$, there must be a nondescent leaf $f$ between $\mathbf{e}$ and $\mathbf{j}$. Consider the step where $d$ was labeled in the caterpillar labeling algorithm. At that step, $d$ and $i$ were the two smallest remaining labels, since $d$ was the one chosen by the algorithm, $i<d$, and neither had yet been used. So, $f>i$. Thus, we can repeat the above argument, replacing $d$ with $f$. Since the tree is finite, we will eventually reach a step where $\mathbf{e}\neq i$, since there is a $\mathbf{j}\neq i$ between $d$ and $\mathbf{i_l}$. Thus, we do not have a valid slide tree.
		
		\textbf{Case 2:} $\mathbf{j}<i$. Moreover, we have $\mathbf{j}<i<d$. We again let $\mathbf{e}$ be the edge to the right of $d$. Since $\mathbf{e}$ cannot be both $i$ and $\mathbf{j}$, when $\mathbf{e}$ slides in the slide algorithm it will see at least one of the leaves $i$ and $j$ to its right, but since it is sliding to $d$ it will compare something smaller than $d$ to $d$, which is not allowed. Thus, we again do not have a valid slide tree.
		
		We have now shown that in both cases we cannot have a valid slide tree, which contradicts our original assumption. Thus, if $\tree(w)\in\Cat^\psi(\kUnd)$, $w$ must satisfy \eqref{eq:psiCat} for all $i$.
		
		($\impliedby$) We induct on how far $\kUnd$ is from being right-justified. More precisely, we induct on $\#\{(i<j)\,|\,k_i\neq 0, k_j=0\}$. The base case is the right justified case given in Theorem \ref{thm:rjust}. For the inductive step, it suffices to show that if the Proposition is true for $\kUnd'=(\ldots,0,l,\ldots )$, then it is also true for $\kUnd=(\ldots,l,0,\ldots)$.
		
		Let $w\in\Av_{\kUnd}(2{-}1{-}2,23{-}\overline{2}{-}1)$, for some $w$ satisfying \eqref{eq:psiCat} for all $i$ where $k_i\neq0$, and where $\kUnd$ is some non-right justified composition with $k_i=0$ and $k_{i-1}\neq0$ for some $i$. Let $\kUnd'=(k_1,\ldots,k_{i-2},0,k_{i-1},k_{i+1},\ldots,k_n)$, and let $w'=(i,i-1)w$ be the word obtained from $w$ by replacing all $i-1$'s with $i$'s. Then, since the letters of $w$ and $w'$ have all the same relative orders, and the only change in \eqref{eq:psiCat} for $w'$ is that for our chosen $i$, the right-hand side is 1 smaller. So, $w'$ satisfies the necessary condition, and moreover, \eqref{eq:psiCat} is strict for the chosen $i$. Thus, by the inductive hypothesis, $T'=\tree(w')\in\Cat^\psi(\kUnd')$.
		
		Next, we will use this to show that $T=\tree(w)\in\Cat^\psi(\kUnd)$. We will first show that the same leaf labels are arranged the same way in $T$ as in $T'$, except possibly for swapping the leaves $i$ and $i-1$. So, consider what could go wrong in the slide algorithm by either leaving the leaves alone or swapping $i-1$ and $i$. If we do not swap the leaves $i$ and $i-1$, then every comparison made in the slide algorithm is the same in both trees, and so all slides are valid. The only possible issue is that $i-1$ may not be able to label the appropriate edges. That is, in $T'$, $i-1$ might be to the left of an edge $\mathbf{i}$, or there may be an edge between the leaf $i-1$ and the rightmost edge $\mathbf{i}$ that is smaller than $i-1$. Alternatively, if we do swap $i$ and $i-1$, then $i-1$ is in the correct position, and every comparison using at most one of $i$ and $i-1$ is the same comparison. So, the only possible issue is if some edge in $T'$ compares $i$ to $i-1$, then $T$ would not be valid caterpillar tree. We consider different cases on where leaf $i-1$ is relative to leaf $i$ and $\mathbf{i_r}$ in $T'$.
		
		\textbf{Case 1:} $i-1$ is between $i$ and $\mathbf{i_r}$ in $T'$. In this case, we do not swap $i$ and $i-1$ when forming $T$. In the slide algorithm, every edge between $\mathbf{i_r}$ and $i$ is contracted before $i$ slides, so every edge between $i-1$ and $\mathbf{i_r}$ is also contracted. Thus, $i-1$ can slide across $\mathbf{i_r}$ in $T$.
		
		\textbf{Case 2:} $i-1$ is to the right of $i$. In this case, we swap $i$ and $i-1$. Since $i-1$ is to the right of $i$ in $T'$, no edge compared the two in $T'$. So, no edge compares $i-1$ and $i$ against each other in $T$ either, and so all comparisons in $T$ are valid as well.
		
		\textbf{Case 3:} $i-1$ is to the left of $\mathbf{i_r}$. In this case, we swap $i$ and $i-1$. We will show that the furthest left that $i-1$ can be in $T'$ is immediately to the left of $\mathbf{i_l}$. By a similar argument to the one we made in the forwards direction of this proof, the number of nondescent leaves right of $\mathbf{i_r}$ is at least the total number of repeats right of $\mathbf{i_r}$, plus the number of nondescent edge labels larger than $i$ (which by Lemma \ref{lem:largerToRight} all occur to the right of $\mathbf{i_r}$), plus one (because there is one more leaf than edge). That is, $\rnd\geq\TRep_w(i)+p_++1$. As noted above, for $w'$ and our chosen $i$, \eqref{eq:psiCat} is strict. So, combining these two inequalities and using Lemma \ref{lem:2leaves}, we have: \begin{align*} \rnd-p_+-1+\ell_w(i) &\geq \TRep_w(i)+\ell_w(i) \geq z_++1\\ \rnd+\ell_w(i) -z_+-p_+ &\geq 2\\ z_-+p_--\lnd+2 &\geq 2\\ z_-+p_- &\geq \lnd. \end{align*} 
		Thus, the number of nondescent leaves left of $\mathbf{i_l}$ is at most the total number of nondescent leaves (excluding $c$) less than $i$. We again note that in our $\tree()$ algorithm, we label the nondescent leaves from left to right, and at each step, we chose one of the two smallest remaining unused labels. So, the first time $i-1$ (the largest nondescent label smaller than $i$) would appear as an option would be the rightmost nondescent leaf left of $\mathbf{i_l}$. If the leaf immediately to the left of $\mathbf{i_l}$ is a nondescent leaf, then we are done. Otherwise, let $v$ be the rightmost nondescent leaf left of $\mbf{i_l}$. Since any leaves between $v$ and $\mathbf{i_l}$ are descent leaves, the edges between them (including $\mbf{i_l}$ itself) are strictly descreasing from left to right. Thus, the edge on the right of $v$ is strictly larger than $i-1$. So, when we label $v$ in Step 7 of Definition \ref{def:tr}, our options are two labels, the larger of which is at most $i-1$, both of whom are smaller than the edge to the right. So, we will choose the smaller of the two, and not label $v$ by $i-1$.
		
		So, we have shown that if $i-1$ is left of $\mathbf{i_r}$, then the edge to the immediate right of $i-1$ is an edge labeled $\mathbf{i}$. So, the (first) label that will compare something to $i-1$ is $i$ itself, which will compare some other label $l\notin\{i-1,i\}$ to $i-1$. After this point, every edge between $i-1$ and $i$ is contracted, so no other label can compare one to the other. If we swap $i-1$ and $i$ in $T$, then instead $i-1$ will compare $l$ to $i$, which is fine for the $\psi$-slide algorithm. Thus, swapping the leaves $i-1$ and $i$ when forming $T$ does not introduce any problems.
		
		Therefore, in any case, we can take the valid caterpillar tree $T'\in\Cat^\psi(\kUnd')$ to form another tree $T$ with edge sequence $w$ without introducing any disallowed slides, meaning $T\in\Cat^\psi(\kUnd)$.
	\end{proof}
	
	\subsection{Pattern avoidance conditions for $\Cat^\omega(\kUnd)$}
	
	In this last subsection, we us our characterization of $\psi$-caterpillars to characterize the $\omega$-caterpillar trees in $\Slwk$, including when $\kUnd$ is not right-justified.

	\begin{definition}
		For a word $w$ with content $\kUnd$, let $k(i,j)$ denote the number of $j$'s in the subword of $w$ right of the rightmost $i$. Define the number of repeats larger than $i$ right of the rightmost $i$ by $$\rep_w(i)=\sum_{j=i+1}^{n}\max((k(i,j)-1),0).$$
	\end{definition}
	
	\begin{lemma}
		Let $\kUnd$ be a reverse-Catalan composition, and let $w$ be a word of content $\kUnd$. If $\rep_w(i)< \z(i)$ for some $i$, then there is some $j$ for which $\rep_w(j)< \z(j)$ and the subword  of $w$ right of the rightmost $j$ contains no entries smaller than $j$.
		\label{lem:rightmost_i}
	\end{lemma}
	
	\begin{proof}
		For such an $i$, let $j$ be the rightmost entry in $w$ (weakly) smaller than $i$. Since $j\leq i$, $\z(j)\geq \z(i)$. Then, since $j$ is the rightmost entry right of $i$, there are no entries $l$ with $j<l<i$ right of $j$. So, everything after $j$ is also larger than $i$, so $\rep_w(j)\leq\rep_w(i)$. Thus, we have $\rep_w(j)\leq \rep_w(i)<\z(i)\leq \z(j)$. The second condition is satisfied by definition of how we chose $j$.
	\end{proof}
	
	Finally, we prove the full classification for the $\omega$ caterpillar trees.
	
	\begin{lemma}
		Let $w$ be a word of content $\kUnd$ and $T=\tree(w)\in\Slpk$. If in Step 7 of Definition \ref{def:tr}, the second smallest label was ever used for a leaf, then $T\notin\Slwk$.
		\label{lem:2nd_smallest_not_omega}
	\end{lemma}
	
	\begin{proof}
		Let $v$ be a leaf that was labeled by the second smallest remaining unused label in Step 7 of Definition \ref{def:tr}, and $\mbf e$ the edge on its right. For this to have happened, we must have $\mbf{e}<v$. However, since $v$ is a nondescent leaf, $\mbf e$ slides before the edge to the left of $v$, and so $\mbf e$ slides to $v$. Since $\mbf e<v$, $\mbf e$ is not a valid $\omega$-slide, and thus $T\notin\Slwk$.
	\end{proof}
	
	As a consequence of this, we know that in any $\omega$ caterpillar tree, the nondescent leaves always increase from left to right.
	
	\begin{lemma}
		Let $T\in\Cat^\psi(\kUnd)$, and $i$ a label such that $k_i\neq0$. Then the number of leaves right of $\mbf{i_r}$ that do not label any edges weakly right of $\mbf{i_r}$ is $\TRep_w(i)+1$.
		\label{lem:label_out_of_branch}
	\end{lemma}
	
	\begin{proof}
		Let $B$ be the branch starting at $\mbf{i_r}$. First, any branch has one more leaf than internal edges. Each distinct edge label in $B$ corresponds to exactly one leaf label. So, the number of leaves that do not label any edges in $B$ is the number of repeated edges in $B$, plus one. That is, there are $\TRep_w(i)+1$ such leaves.
	\end{proof}
	
	Now, we prove the $\omega$ version of Theorem \ref{thm:psiAvoidance}.
	
	\begin{theorem}
		Let $\kUnd$ be a reverse-Catalan composition, and let $w$ be a word of composition $\kUnd$. Then, $\tree(w)\in\Cat^\omega (\kUnd)$ if and only if $w\in\Av_{\kUnd}(2{-}1{-}2, 23{-}\overline{2}{-}1)$ and \begin{equation}
		\rep_w(i)\geq \z(i)
		\label{eq:omegaCat}
		\end{equation}
		for all $i$ such that $k_i>0$.
		\label{thm:omegaAvoidance}
	\end{theorem}
	
	\begin{proof}
		($\implies$) The avoidance condition is satisfied by Lemma \ref{lem:avoids} and the fact that $\Cat^\omega(\kUnd)\subseteq\Cat^\psi (\kUnd)$. 	Next, assume that $\tree(w)$ is a valid $\omega$-slide tree. Suppose, aiming for contradiction, that $\rep_w(i)< \z(i)$ for some $i$ with $k_i\neq0$. Then by Lemma \ref{lem:rightmost_i}, there is some $j$ for which $\rep_w(j)< \z(j)$ and there are no entries smaller than $j$ to the right of $j$ in $w$. We will show that $j$ cannot perform an $\omega$ $\kUnd$-slide.
		
		For $j$ to be able to perform an $\omega$ slide, it is necessary for there to be a leaf $m$ to the right of the edge $\mbf j$ with $m<\mbf j$. 
		Such an $m$ would be a nondescent leaf, since there is no edge $m<\mbf j$ to the right of $j$. Note that for any $l$, if $k_l=0$, then $l$ is a nondescent leaf. 
		This means that at least $\z(j)$ nondescent leaf labels are larger than $j$. By Lemma \ref{lem:2nd_smallest_not_omega}, this means the $z(j)$ rightmost nondescent leaves are all larger than $j$. Since $\mbf{j_r}$ has no edges smaller than $j$ to its right, $\rep_w(j)=\TRep_w(j)$. So, by Lemma \ref{lem:label_out_of_branch}, there are $\rep_w(j)+1$ leaves right of $\mbf{j_r}$ that do not label an edge weakly right of $\mbf{j_r}$. Since we assumed that $\z(j)\geq \rep_w(j)+1$, there is no room for a nondescent label $m<j$ to the right of $\mbf{j_r}$, and so $j$ cannot perform an $\omega$-slide. Thus, $\tree(w)\notin\Cat^\omega(\kUnd)$, and this concludes the forwards implication. 
		
		($\impliedby$) Suppose that $w\in\Av_{\kUnd}(2{-}1{-}2, 23{-}\overline{2}{-}1)$ and $\rep_w(i)\geq \z(i)$ for all $i$ where $k_i\neq0$. Since \eqref{eq:omegaCat} is a stronger condition than \eqref{eq:psiCat}, Theorem \ref{thm:psiAvoidance} gives us that $\tree(w)\in\Cat^\psi(\kUnd)$. 
		We want to show that $\tree(w)\in\Cat^\omega(\kUnd)$ as well. To do this, it suffices to show that for each $j$ that labels an edge, $\mbf{j_r}$ is a valid $\omega$-slide, since that ensures that $j$ is larger than the largest thing it ever compares in the slide algorithm.
		
		By Lemma \ref{lem:label_out_of_branch}, there are $\TRep_w(j)+1$ leaf labels right of $\mbf{j_r}$ that do not label an edge weakly right of $\mbf{j_r}$. If any of these leaves label an edge to the left of $j$, they must be smaller than $j$, since $\tree(w)\in\Slpk$. So, any of these $\TRep_w(j)+1$ leaves that are larger than $j$ must not label any edges at all. Since there are exactly $\z(j)$ of these in the whole tree, and $\TRep_w(j)+1\geq\rep_w(j)+1>\z(j)$, there is at least one leaf to the right of $j$ smaller than $j$, so $j$ is a valid $\omega$-slide. Thus, $\tree(w)\in\Cat^\omega(\kUnd)$.
	\end{proof}

	\section{The case $\kUnd=(1,1,\ldots,1)$}
	
	We now show that in the case $\kUnd=(1,1,1,\ldots,1)$, we can express the iterative bijection more directly. Recall Proposition \ref{prop:catAvoid23-1}, in which \cite{GGL22} define a bijection between $\Cat(1,1,\ldots,1)$ and the set of $23{-}1$-avoiding permutations. Our work in this section will extend this map to a full bijection between $\Slwl$ and $S_n$.
	
	\begin{definition}
		Let $\tau\in S_n\setminus\Av_n(23{-}1)$. We call $x,y,z$ the \textbf{earliest} instance of $23{-}1$ if $\tau=wxyuzv$, where $w$, $u$, and $v$ are (possibly empty) words, $x$, $y$, and $z$ form a $23{-}1$ pattern, and $\tau$  has no instances of $23{-}1$ that either use letters from $w$, or use $x$, $y$, and a letter from $u$.
	\end{definition}
	
	\begin{lemma}
		Let $B$ be a branch of a slide tree $T\in\Slwl$. Then, all leaves of $B$ except $\min(B)$ slide within $B$.
		\label{lem:minNotSlide}
	\end{lemma}
	
	\begin{proof}
		We first show that $\min(B)$ does not slide within $B$. Suppose instead that $\min(B)$ did slide along some edge $\mbf e$. Then, let $B_\mbf e$ be the branch starting at $\mbf e$, and $i$ be the number of leaves of $B_\mbf e$. Then, $B_\mbf e$ has $i-1$ internal edges, including $\mbf e$. For $\min(B)$ to slide along $\mbf e$, all $i-1$ larger leaf labels must slide before $\min(B)$. However, there are only $i-2$ other edges in $B_\mbf e$ available for them to slide on in order to leave $\mbf e$ unlabeled. This is an impossibility, and so $\min(B)$ cannot slide within $B$.
		
		Second, we show that all leaves other than $\min(B)$ slide in $B$. Each internal edge of $B$ must be labeled by some leaf in $B$, each leaf can can be used for at most one edge (since $\kUnd=(1,1, \ldots,1)$), and we just showed that $\min(B)$ does not slide in $B$. So, since there is one more leaf than internal edge in $B$, all leaves other than $\min(B)$ must slide within $B$.
	\end{proof}
	
	\begin{lemma}
		Let $T\in\Slwk$, where $\kUnd$ is the all 1's composition of $n$. Let $\mbf e$ be an internal edge of $T$, $B$ the branch of $T$ starting at $\mbf e$, and $i+1$ be the number of leaves of $B$. Construct $T'$ to be the tree resulting from $T$ by replacing $B$ with a single leaf $\min(B)$. Next, relabel the leaves of $T'$ so that every label $a,b,c,1,2,\ldots,n-i$ is used once, while keeping the ordering of the leaves the same. Then,  $T'\in\Slide^\omega(\kUnd')$, where $\kUnd'$ is the all 1's composition of $n-i$.
		\label{lem:removebranch}
	\end{lemma}
	
	\begin{proof}
		We will show that the slide labeling algorithm completes successfully for $T'$. We do this by showing that each edge is a valid slide. 
		
		First, any internal edge of $T$ in $B$ does not exist in $T'$, so there is no corresponding slide to check in $T'$. Otherwise, consider some edge $\mbf e$ in $T\setminus B$ that is not labeled $\min(B)$. In $T$, the comparison made in the slide algorithm at step $\mbf e$ either does not consider any of $B$, or considers a branch $C$ containing $B$. The corresponding branch $C'$ in $T'$ still contains $\min(B)$, so $\min(C)=\min(C')$. Thus, the slide labeling $\mbf e$ is valid in $T'$. 
		
		Finally, we consider the label $\min(B)$. If $\min(B)=c$, then there is no edge $\min(B)$, and we are done. Otherwise, $\min(B)$ slides in $T$ after all other leaves of $B$. So, it slides after the internal edges of $B$ are all contracted. So in $T'$, since $\min(B)$ is a leaf where $B$ was removed, it will label the same edge and compare the same labels as it did in $T$. Thus, every edge in $T'$ is a valid slide, so $T'\in\Slide^\omega(\kUnd')$.
	\end{proof}
	
	Next, we define the map $\phi$ that tells us how to read a word (a permutation, in fact) off from the edge labels of a $(1,1,\ldots,1)$-slide tree.
	
	\begin{definition}
		We define the map $\phi:\Slwl\to S_n$ as follows. Consider a slide tree $T\in\Slwl$ along with its edge labels given by the slide labeling algorithm. Then, let $T'$ be the tree obtained by deleting the leaves and leaf edges of $T$, rooted at the vertex  adjacent to $a$ in $T$. (In other words, we record the slide labels and then delete all leaves and non-internal edges.) Note that $T'$ is an at-most-trivalent tree. 
		
		Next, starting at the root of $T'$, read off and record the edge labels of $T'$ in order from left to right. When a degree 3 vertex is reached, recursively apply this same process to both branches $B$ and $B'$, where $B$ is the branch with larger minimal label. Then, take the words obtained from $B$ and $B'$ and concatenate them in that order to the word obtained thus far for $T'$. Since the content of the edge labels is $(1,1,\ldots,1)$, and this process reads off each edge exactly once, the resulting word is an element of $S_n$.
		\label{def:phi}
	\end{definition}
	
	\begin{lemma}
		If $T\in\Slwl$ is not a caterpillar, then $\phi(T)$ contains the pattern $23{-}1$.
		\label{lem:has23-1}
	\end{lemma}
	
	\begin{proof}
		Let $v$ be a vertex in $T$ adjacent to three internal edges $\mbf x$, $\mbf y$, and $\mbf z$, where $\mbf x$ is the edge towards $a$. Without loss of generality, we assume that $\mbf z<\mbf y$. By Corollary \ref{cor:only212or231}, we have $\mbf z<\mbf x<\mbf y$. Let $U$ be the branch starting at the edge $\mbf y$, and $V$ the branch starting at $\mbf z$.
		
		Next, since $\mbf x>\mbf z$, the label $\mbf x$ must come from $U$. Since by Lemma \ref{lem:minNotSlide} every leaf in $U$ slides within $U$ except for $\min(U)$, and the edge $\mbf x$ is not in $U$, $\mbf x=\min(U)$. Since $\mbf x$ is a valid slide, $\min(V)<\min(U)$, so the map $\phi$ will read off the edges of $U$ before $V$. Thus, $\phi(T)=\ldots xy\ldots z\ldots$, and $x$, $y$, and $z$ will form a $23{-}1$ pattern.
	\end{proof}
	
	\begin{definition}
		Define $EL_n$ to be the set of rooted \textbf{edge-labeled trivalent trees} with $n$ internal edges labeled by $[n]$ (and unlabeled leaves).
	\end{definition}
	
	We define the map $\nu:\Av_n(23{-}1)\to\Cat^\omega(1,1,\ldots,1)$ to be the restriction of the map $\tree()$ from Section \ref{sec:cats} to permutations that avoid the pattern $23{-}1$. Note that this map coincides with the \textit{leaf labeling algorithm} defined in Definition 6.3 of \cite{GGL22}, and that each tree in its image is a valid caterpillar slide tree.
	
	We wish to extend the map $\nu$ to a map $\rho:S_n\to\Slwl$. We first define a map $\hat\rho:S_n\to EL_n$, and then show that each tree in the image $\hat\rho$ has a unique leaf labeling that obtains its edge labels under the slide labeling algorithm, and thus $\hat{\rho}$ induces the map $\rho$.
	
	\begin{definition}
		Define $\hat\rho:S_n\to EL_n$ as follows. If $\tau\in\Av_n(23{-}1)$, we let $\hat\rho(\tau)=\nu(\tau)$ (that is, the caterpillar tree whose edge labels read from left to right as $\tau$). Otherwise, let $x,y,z$ be the \textit{earliest} $23{-}1$ pattern in $\tau$, and $w,u,v$ be words such that $\tau=wxyuzv$. Then we construct $\hat\rho(\tau)$ as follows. Start with the caterpillar subtree $\hat\rho(wx)$. Then, to the internal vertex furthest from the root, we replace the two external edges with the two subtrees $\hat\rho(yu)$ and $\hat\rho(zv)$.
		\label{def:psi}
	\end{definition}
	
	\begin{notation}
		Throughout this section, we use the abuse of notation $\hat\rho(w)$, where $w$ is a subword of a permutation, to denote the tree formed by applying $\hat\rho$ to the reduction $\red(w)$, then relabeling the result to have the same labels as the entries of $w$.
	\end{notation}
	
	\begin{definition}
		Given $\tau\in S_n$, define $\rho(\tau)$ to be the unique element of $\Slwl$ whose $(1,1,\ldots, 1)$-slide labeling admits the same edge labels as $\hat\rho(\tau)$.
		\label{def:psi2}
	\end{definition}
	
	\begin{theorem}
		For any $\tau\in S_n$, there is exactly one tree of $\Slwl$ that admits the edge labeling $\hat\rho(\tau)$. In other words, $\rho:S_n\to\Slwl$ is well-defined.
		\label{thm:welldefined}
	\end{theorem}
	
	\begin{proof}
		We proceed by strong induction on the number of branching vertices of $\hat{\rho}(\tau)$ (that is, vertices that are not adjacent to any leaves). The base case is when $\hat{\rho}(\tau)$ has no internal vertices. That is, when $\hat{\rho}(\tau)$ is a caterpillar tree. By Proposition \ref{prop:catAvoid23-1}, there is a unique leaf labeling that gives this slide labeling, so, $\rho(\tau)$ is well-defined in this case.
		
		Otherwise, suppose $\hat{\rho}(\tau)$ has $m$ branching vertices. For the inductive hypothesis, we assume that $\rho(\tau')$ is well-defined whenever $\hat{\rho}(\tau')$ has $m-1$ or fewer branching vertices. Since $\hat\rho(\tau)$ has a branching vertex, $\tau$ contains the pattern $23{-}1$. So, let $\tau=wxyuzv$, where $x,y,z$ is the earliest $23{-}1$ pattern in $\tau$. We will show that $T=\hat\rho(\tau)$ is a slide tree.
		
		In particular, let $A=\hat\rho(wxyu)$ and $B=\hat\rho(wxzv)$. Let $U$ and $V$ be the subtrees formed under $\hat\rho$ by $yu$ and $zv$, respectively. Since $x,y,z$ was the earliest $23{-}1$ pattern in $\tau$, there is no $23{-}1$ pattern in $wxyu$ that involves both $x$ and $y$, and since $z<x$, there is no $23{-}1$ pattern in $wxzv$ that involves both $x$ and $z$. This means that all $23{-}1$ patterns in $wxyu$ or $wxzv$ occur entirely within $yu$ or $zv$ respectively. So, neither $A$ nor $B$ have a branching vertex at the end of $\mbf x$, and thus $A$ (respectively $B$) have the shape of $T$ with $V$ (respectively $U$) removed. By our inductive hypothesis, both $A$ and $B$ are valid slide trees, so we consider them with their associated leaf labels. 
		
		Let $d$ be the leaf of $A$ in the position of $V$ in $T$, and $e$ the leaf of $B$ in the position of $U$ in $T$. Next, we apply Lemma \ref{lem:removebranch} to $A$ and replace the branch $U$ with $\min(U)$ to create a new slide tree $T_A$. Similarly, we apply Lemma \ref{lem:removebranch} to $B$ and replace $V$ with $\min(V)$ to form $T_B$. 
		As a consequence of Lemma \ref{lem:minNotSlide} and Lemma \ref{lem:removebranch}, $T_A$ and $T_B$ are caterpillar trees with the same edge labels. 
		Thus, their leaf labels are also the same by the inductive hypothesis. In particular, consider the leaves at the right ends of $T_A$ and $T_B$. For $T_A$, we already had $d$, and then replaced $U$ with $\min(U)$. Similarly for $T_B$, we have $e$ and $\min(V)$. So as sets, $\{d,\min(U)\}=\{e,\min(V)\}$.
		
		Next, we show that $e=\min(U)$ and $d=\min(V)$. Since $\mbf x>\mbf z$, $\mbf x$ slides before $\mbf z$ in $B$, so $\mbf x=e$ in order for $e$ to slide along edge $\mbf x$. Since $\mbf y$ is a valid slide in $A$ and $\mbf y>\mbf x$, $\min(U)>d$, and so $\min(U)$ will slide before $d$ in $A$. From Lemma \ref{lem:minNotSlide}, $\min(U)$ cannot slide within $U$, so it must slide across the edge $\mbf x$. Thus, $\min(U)=\mbf x=e$. By process of elimination, $d=\min(V)$.
		
		Thus far, we have shown that $A$ and $B$ each have leaf labelings that make them valid slide trees, these labelings agree on $T_A$ and $T_B$, $e=\min(U)$, and $d=\min(V)$. We will now use these facts to build a valid leaf labeling for $T$. For any leaf in $U$ or $V$, we label that leaf by the same label that leaf has in $A$ or $B$, respectively. The remaining leaves all lie between $a$ and $\mbf x$, so we give them the corresponding labels from $T_A=T_B$. This uses all the labels $c,1,2,\ldots,n$ exactly once.
		
		Next, we show that such a leaf labeling gives a slide labeling matching $T$. For any edge labeled neither $d=\min(V)$ nor $e=\min(U)$, and not in the branch $V$, the comparison made at that step by the slide-labeling algorithm will at most only look at a branch containing $V$, and will make the same comparison it made in $A$, where all of $V$ was replaced by $d$. So, the validity of this slide carries over from $A$. Similarly, any slide other than $d$ or $e$ that is outside $U$ is valid because it makes the same comparison as in $B$. Lastly, if there is an edge labeled $d$ (resp. $e$), then $d$ will slide only after the other leaves of $V$ (resp. $U$) have already done so. So, it will slide the same way as it did in $A$ (resp. $B$), when it started where $V$ (resp. $U$) connected to $T$. These cases cover every edge of $T$, so $T$ is a valid slide tree. 
		
		Finally, we must demonstrate uniqueness. Suppose instead that there were two slide trees $T_1,T_2\in\Slwl$ that both had the slide labeling $\hat\rho(\tau)$. Since the edge labels are the same for both $T_1$ and $T_2$, let $A_i,B_i,U_i,V_i$ be defined as above for $i=1,2$. Then, applying Lemma \ref{lem:removebranch} to $U_1$ and $U_2$, by our inductive hypothesis the results are the same tree, so $\min(U_1)=\min(U_2)$, and all other leaves of $A_1$ and $A_2$ must be the same as well. However, we can do the same for $\min(V_1)$, $\min(V_2)$, $B_1$, and $B_2$. Thus, all leaves of $T_1$ and $T_2$ are the same, so in fact $T_1=T_2$. Therefore, the tree $T$ found above is unique, and $\rho$ is well-defined.
	\end{proof}
	
	\begin{example}
		Let $\tau=853769421$. Then, $\hat\rho(\tau)$ is the tree:
		\begin{center}
			\begin{tikzpicture}[scale=0.8]
			\draw (0,0)--(1,0)--(2,0)--(3,0)--(4,0)--(5,0)--(6,0);
			\draw (1,0)--(1,-0.6);
			\draw (2,0)--(2,-0.6);
			\draw (4,0)--(4,-0.6);
			\draw (3,0)--(3,-1)--(3,-2);
			\draw (3,-1)--(2.4,-1);
			\draw (3,-2)--(2.5,-2.5);
			\draw (3,-2)--(3.5,-2.5);
			\draw (5,0)--(5,-1);
			\draw (5,-1)--(4.5,-1.5);
			\draw (5,-1)--(5.5,-1.5);
			\draw (0,0)--(-0.5,0.5);
			\draw (0,0)--(-0.5,-0.5);
			\draw (6,0)--(6.5,0.5);
			\draw (6,0)--(6.5,-0.5);
			\fill (0,0) circle (0.1);
			
			\node[] at (0.5,0.2) {$8$};
			\node[] at (1.5,0.2) {$5$};
			\node[] at (2.5,0.2) {$3$};
			\node[] at (3.5,0.2) {$7$};
			\node[] at (4.5,0.2) {$6$};
			\node[] at (5.5,0.2) {$9$};
			
			\node[] at (3.2,-0.5) {$2$};
			\node[] at (3.2,-1.5) {$1$};
			\node[] at (5.2,-0.5) {$4$};
			\end{tikzpicture}
		\end{center}
		The earliest instance of $23{-}1$ in $\tau$ is the subword $372$, so $w=85$, $x=3$, $y=7$, $u=694$, $z=2$, and $v=1$. Then, the trees $A=\hat{\rho}(txyu)$ and $B=\hat{\rho}(txzv)$ are the trees drawn below, respectively.
		\begin{center}
			\begin{tikzpicture}
			\draw (0,0)--(1,0)--(2,0)--(3,0)--(4,0)--(5,0)--(6,0);
			\draw (1,0)--(1,-0.6);
			\draw (2,0)--(2,-0.6);
			\draw (4,0)--(4,-0.6);
			\draw (3,0)--(3,-0.6);
			\draw (5,0)--(5,-1);
			\draw (5,-1)--(4.5,-1.5);
			\draw (5,-1)--(5.5,-1.5);
			\draw (0,0)--(-0.5,0.5);
			\draw (0,0)--(-0.5,-0.5);
			\draw (6,0)--(6.5,0.5);
			\draw (6,0)--(6.5,-0.5);
			\fill (0,0) circle (0.1);
			
			\node[] at (0.5,0.2) {$8$};
			\node[] at (1.5,0.2) {$5$};
			\node[] at (2.5,0.2) {$3$};
			\node[] at (3.5,0.2) {$7$};
			\node[] at (4.5,0.2) {$6$};
			\node[] at (5.5,0.2) {$9$};
			
			\node[] at (5.2,-0.5) {$4$};
			\end{tikzpicture}
			\hspace{1cm}
			\begin{tikzpicture}
			\draw (0,0)--(1,0)--(2,0)--(3,0)--(3.5,0.5);
			\draw (1,0)--(1,-0.6);
			\draw (2,0)--(2,-0.6);
			\draw (3,0)--(3,-1)--(3,-2);
			\draw (3,-1)--(2.4,-1);
			\draw (3,-2)--(2.5,-2.5);
			\draw (3,-2)--(3.5,-2.5);
			\draw (0,0)--(-0.5,0.5);
			\draw (0,0)--(-0.5,-0.5);
			\fill (0,0) circle (0.1);
			
			\node[] at (0.5,0.2) {$8$};
			\node[] at (1.5,0.2) {$5$};
			\node[] at (2.5,0.2) {$3$};
			
			\node[] at (3.2,-0.5) {$2$};
			\node[] at (3.2,-1.5) {$1$};
			\end{tikzpicture}
		\end{center}
		Then, inductively we know that $A$ and $B$ have unique leaf labels that give these edge labelings. 
		\begin{center}
			\begin{tikzpicture}
			\draw (0,0)--(1,0)--(2,0)--(3,0)--(4,0)--(5,0)--(6,0);
			\draw (1,0)--(1,-0.6);
			\draw (2,0)--(2,-0.6);
			\draw (4,0)--(4,-0.6);
			\draw (3,0)--(3,-0.6);
			\draw (5,0)--(5,-1);
			\draw (5,-1)--(4.5,-1.5);
			\draw (5,-1)--(5.5,-1.5);
			\draw (0,0)--(-0.5,0.5);
			\draw (0,0)--(-0.5,-0.5);
			\draw (6,0)--(6.5,0.5);
			\draw (6,0)--(6.5,-0.5);
			\fill (0,0) circle (0.1);
			
			\node[] at (0.5,0.2) {$8$};
			\node[] at (1.5,0.2) {$5$};
			\node[] at (2.5,0.2) {$3$};
			\node[] at (3.5,0.2) {$7$};
			\node[] at (4.5,0.2) {$6$};
			\node[] at (5.5,0.2) {$9$};
			
			\node[] at (5.2,-0.5) {$4$};
			
			\node[] at (-0.6,0.6) {$a$};
			\node[] at (-0.6,-0.6) {$b$};
			\node[] at (1,-0.8) {$8$};
			\node[] at (2,-0.8) {$5$};
			\node[] at (3,-0.8) {$c$};
			\node[] at (4,-0.8) {$7$};
			\node[] at (4.4,-1.6) {$3$};
			\node[] at (5.6,-1.6) {$4$};
			\node[] at (6.6,0.6) {$9$};
			\node[] at (6.6,-0.6) {$6$};
			\end{tikzpicture}
			\hspace{1cm}
			\begin{tikzpicture}
			\draw (0,0)--(1,0)--(2,0)--(3,0)--(3.5,0.5);
			\draw (1,0)--(1,-0.6);
			\draw (2,0)--(2,-0.6);
			\draw (3,0)--(3,-1)--(3,-2);
			\draw (3,-1)--(2.4,-1);
			\draw (3,-2)--(2.5,-2.5);
			\draw (3,-2)--(3.5,-2.5);
			\draw (0,0)--(-0.5,0.5);
			\draw (0,0)--(-0.5,-0.5);
			\fill (0,0) circle (0.1);
			
			\node[] at (0.5,0.2) {$8$};
			\node[] at (1.5,0.2) {$5$};
			\node[] at (2.5,0.2) {$3$};
			
			\node[] at (3.2,-0.5) {$2$};
			\node[] at (3.2,-1.5) {$1$};
			
			\node[] at (-0.6,0.6) {$a$};
			\node[] at (-0.6,-0.6) {$b$};
			\node[] at (1,-0.8) {$8$};
			\node[] at (2,-0.8) {$5$};
			\node[] at (2.2,-1) {$2$};
			\node[] at (3.6,0.6) {$3$};
			\node[] at (2.4,-2.6) {$c$};
			\node[] at (3.6,-2.6) {$1$};
			\end{tikzpicture}
		\end{center}
		Then, we can piece the leaf labels back together.
		\begin{center}
			\begin{tikzpicture}
			\draw (0,0)--(1,0)--(2,0)--(3,0)--(4,0)--(5,0)--(6,0);
			\draw (1,0)--(1,-0.6);
			\draw (2,0)--(2,-0.6);
			\draw (4,0)--(4,-0.6);
			\draw (3,0)--(3,-1)--(3,-2);
			\draw (3,-1)--(2.4,-1);
			\draw (3,-2)--(2.5,-2.5);
			\draw (3,-2)--(3.5,-2.5);
			\draw (5,0)--(5,-1);
			\draw (5,-1)--(4.5,-1.5);
			\draw (5,-1)--(5.5,-1.5);
			\draw (0,0)--(-0.5,0.5);
			\draw (0,0)--(-0.5,-0.5);
			\draw (6,0)--(6.5,0.5);
			\draw (6,0)--(6.5,-0.5);
			\fill (0,0) circle (0.1);
			
			\node[] at (0.5,0.2) {$8$};
			\node[] at (1.5,0.2) {$5$};
			\node[] at (2.5,0.2) {$3$};
			\node[] at (3.5,0.2) {$7$};
			\node[] at (4.5,0.2) {$6$};
			\node[] at (5.5,0.2) {$9$};
			
			\node[] at (3.2,-0.5) {$2$};
			\node[] at (3.2,-1.5) {$1$};
			\node[] at (5.2,-0.5) {$4$};
			
			\node[] at (-0.6,0.6) {$a$};
			\node[] at (-0.6,-0.6) {$b$};
			\node[] at (1,-0.8) {$8$};
			\node[] at (2,-0.8) {$5$};
			\node[] at (4,-0.8) {$7$};
			\node[] at (4.4,-1.6) {$3$};
			\node[] at (5.6,-1.6) {$4$};
			\node[] at (6.6,0.6) {$9$};
			\node[] at (6.6,-0.6) {$6$};
			\node[] at (2.2,-1) {$2$};
			\node[] at (2.4,-2.6) {$c$};
			\node[] at (3.6,-2.6) {$1$};
			\end{tikzpicture}
		\end{center}
		Thus, $\rho(\tau)\in\Slide^\omega(1,1,1,1,1,1,1,1,1)$.
		
	\end{example}
	
	\begin{theorem}
		The maps $\phi$ and $\rho$ are inverse bijections.
	\end{theorem}
	
	\begin{proof}
		Since $S_n$ and $\Slwl$ are finite sets of the same cardinality, it suffices to show that $\phi(\rho(\tau))=\tau$ for all $\tau\in S_n$.
		
		We do induction on the number of internal vertices of $\rho(\tau)$. If $\tau\in\Av_n(23{-}1)$, then $\phi(\tau)$ is a caterpillar tree with edge labels in the same order as the letters in $\tau$. So, $\phi$ starts at the root of $\rho(\tau)$ and reads off the edges in order, recovering $\tau$.
		
		Otherwise, let $\tau=wxyuzv$, such that $x,y,z$ is the earliest instance of $23{-}1$. Then, $\rho(\tau)$ can be partitioned into the subtrees $\rho(wx)$, $\rho(yu)$, and $\rho(zv)$, where $\rho(wx)$ is caterpillar and $\rho(yu)$ and $\rho(zv)$ are rooted at the end of $\rho(wx)$. Since $x$ is the minimal leaf of $\rho(yu)$ and $z<x$, we will read $\rho(yu)$ before $\rho(zv)$ when we apply $\phi$ to $\rho(\tau)$. Since $\rho(yu)$ and $\rho(zv)$ have strictly fewer internal vertices, by are inductive hypothesis $\phi(\rho(yu))=yu$ and $\phi(\rho(zv))= zv$. So, $$\phi(\rho(\tau))= wx\phi(\rho(yu))\phi(\rho(zv))=wxyuzv=\tau$$ as desired.
	\end{proof}

	\bibliography{refs}
	\bibliographystyle{plain}
	
\end{document}